\documentclass[11pt]{article}
\usepackage{latexsym,color,amsmath,amsthm,amssymb,amscd,amsfonts}
\usepackage{amsmath}
\usepackage{comment}
\usepackage[matrix,arrow,ps]{xy}
\usepackage{tikz-cd}
\usepackage[affil-it]{authblk}
\newtheorem{lemma}{Lemma}[section]
\newtheorem{proposition}[lemma]{Proposition}
\newtheorem{remark}[lemma]{Remark}

\newtheorem{example}[lemma]{Example}

\def\R{{\mathbb R}}

\def\upddots{\mathinner{\mkern 1mu\raise 1pt \hbox{.}\mkern 2mu
\mkern 2mu \raise 4pt\hbox{.}\mkern 1mu \raise 7pt\vbox {\kern 7
pt\hbox{.}}} }

\newcommand{\spn}{{Sp_{2n}(\F)}}

\newcommand{\slt}{{\rm{SL}_{2}({\rm F})}}
\newcommand{\sltt}{{\rm{G}}}
\newcommand{\mslt}{\widetilde{\rm{G}}}

\newcommand{\F}{{\rm F}}

\newcommand{\Of}{\mathbb O_{\F}}
\newcommand{\Pf}{\mathbb P_{\F}}

\newcommand{\mb}{{\widetilde{B}}}

\newcommand{\N}{\mathbb N}
\newcommand{\Z}{\mathbb Z}

\newcommand{\half}{\frac{1}{2}}

\newcommand{\ab} {|\!|}

\newcommand{\Q}{\mathbb Q}
\newcommand{\C}{\mathbb C}

\def\>{\rangle}
\def\<{\langle}

\newtheorem{lem}[lemma]{Lemma}
\newtheorem{thm}[lemma]{Theorem}
\newtheorem{cor}[lemma]{Corollary}

\newtheorem{deff}[lemma]{Definition}
\newtheorem{prop}[lemma]{Proposition}

\numberwithin{equation}{section}

\newcommand{\msl}{\widetilde{{\rm SL}_2(\F)}}

\def\dotunion{
\def\dotunionD{\bigcup\kern-9pt\cdot\kern5pt}
\def\dotunionT{\bigcup\kern-7.5pt\cdot\kern3.5pt}
\mathop{\mathchoice{\dotunionD}{\dotunionT}{}{}}} \date{}
\newcommand\blfootnote[1]{%
  \begingroup
  \renewcommand\thefootnote{}\footnote{#1}%
  \addtocounter{footnote}{-1}%
  \endgroup
}
\makeatletter
\def\@maketitle{%
  \newpage
  \null
  \vskip 2em%
  \begin{center}%
  \let \footnote \thanks
    {\Large\bfseries \@title \par}%
    \vskip 1.5em%
    {\normalsize
      \lineskip .5em%
      \begin{tabular}[t]{c}%
        \@author
      \end{tabular}\par}%
    \vskip 1em%
    {\normalsize \@date}%
  \end{center}%
  \par
  \vskip 1.5em}
\makeatother

\begin{document}
\title{On Shahidi local coefficients matrix}
\author{Dani Szpruch
  \thanks{Correspondence to be sent to: \texttt{dszpruch@openu.ac.il}}}
\affil{Department of Mathematics and Computer Science, Open University of Israel, 43107 Raanana, Israel.}
\maketitle
\centerline {\large \it To Professor  Freydoon Shahidi on his 70th birthday}
\begin{abstract}
In this article we define and study the Shahidi local coefficients matrix associated with a genuine principal series representation ${\rm I}(\sigma)$ of  an $n$-fold cover of $p-$adic $\slt$ and an additive character $\psi$. The conjugacy class of this matrix is an invariant of the inducing representation $\sigma$ and $\psi$ and its entries are linear combinations of Tate or Tate type $\gamma$-factors. We relate these entries to functional equations associated with linear maps defined on the dual of the space of Schwartz functions. As an application we give new formulas for the Plancherel measures and use these to relate  principal series representations of different  coverings of $\slt$. While we do not assume that the residual characteristic of ${\rm F}$ is relatively prime to $n$ we do assume that $n$ is not divisible by 4.
\end{abstract}
{\bf \large Acknowledgements.}

We would like to thank Fan Gao for his remarks on earlier versions of the manuscript. We have essentially an equal role in writing Section \ref{fangao}. We would  also like to thank Alexander Burstein, Francois Ramaroson, Sankar Sitaraman and Valentin Buciumas for useful discussions on the subject matter. Finally we would like to thank the referee for numerous suggestions which significantly  improved the style and clarity of this paper. At the time this manuscript was prepared, the author was partially supported by a Simons Foundation Collaboration Grant 426446.

\section{Introduction}
Let ${\rm F}$ be a $p$-adic field containing the full group of $n^{th}$ roots of 1 and let $\widetilde{\rm{G}}^{(n)}=\mslt$ be the $n$-fold cover of ${\rm G}=\slt$ constructed by Kubota in \cite{Kub}. For a genuine principal series representation of $\mslt$  and an additive character of ${\rm F}$ we associate what we call an Slcm, a Shahidi local coefficients matrix. We compute it and use it to study the Plancherel measure. The study contained in this paper is a generalization  and extension of  the study initiated in \cite{GoSz}. While we do not assume that the residual characteristic of ${\rm F}$ is relatively prime to $n$ we do assume for most of this paper that $n$ is not divisible by 4. To ease the exposition in this introduction we first discuss our results in the case where $n$ is odd.

A genuine principal series representation of $\mslt$ is a representation parabolically induced from $(\sigma,V)$, a genuine smooth irreducible  representation of  $\widetilde{\rm{H}}$, the inverse image of the diagonal subgroup of $\sltt$ inside $\mslt$. The isomorphism class of $\sigma$ is determined by its central character, $\chi_\sigma$. On the other hand, $\chi_\sigma$ is determined by a restriction of a character $\chi$ of ${\rm F^*}$ to ${{\rm F^*}}^n$. Let ${\rm I}(\sigma_s)$ be a genuine principal series representation of $\mslt$ induced from $\sigma_s$ (here $s \in \C$ is the usual complex parameter), let $\psi$ be a non-trivial character of ${\rm F}$  and  let $\operatorname{Wh}_\psi \bigl({\rm I}(\sigma_s)\bigr)$ be the finite dimensional space of $\psi$-Whittaker functionals on ${\rm I}(\sigma_s)$.
Let $${\rm A}_w(\sigma_s):{\rm I}\bigl(\sigma_s \bigr) \rightarrow{\rm I}\bigl((\sigma_s)^w \bigr)$$ be the standard intertwining operator. By duality, ${\rm A}_w(\sigma_s)$ induces a map
$${\rm A}_w^\psi(\sigma_s):\operatorname{Wh}_\psi \bigl({\rm I}\bigl((\sigma_s)^w \bigr)\bigr)\rightarrow \operatorname{Wh}_\psi \bigl({\rm I}(\sigma_s)\bigr).$$
In this paper we define the Slcm associated with $\sigma$ and $\psi$ to be a certain matrix representing ${\rm A}_w^\psi(\sigma_s)$. We use the fact that an integral of Jacquet-type identifies both $\operatorname{Wh}_\psi \bigl({\rm I}\bigl((\sigma_s)^w \bigr)\bigr)$ and $\operatorname{Wh}_\psi \bigl({\rm I}(\sigma_s)\bigr)$ with $V^*$, the space of linear functionals on $V$, to show that the conjugacy class of the matrix defined here is an invariant of $\sigma$ and $\psi$. In particular, its trace and determinant, denoted by $T(\sigma,s,\psi)$ and $D(\sigma,s,\psi)$ respectively are well defined invariants of $\sigma$ and $\psi$.

In the case where $n$ is prime to the residual characteristic of ${\rm F}$,  a closely related matrix associated with unramified representations of coverings of ${\rm GL}_n({\rm F})$ was computed by Kazhdan and Patterson in \cite{KP} and utilized for the study of the distinguished representations, see also \cite{Gao16}. In this context this matrix is sometimes called a scattering matrix.  McNamara generalized in \cite{Mc2} the computations of this matrix to the context of unramified representations of unramified coverings of reductive $p$-adic groups. In recent years this matrix appears frequently in the study of metaplectic groups. For example, it appears in the construction given by Chinta-Offen in \cite{CO} and by McNamara in \cite{Mc2}, of the metaplectic Casselman-Shalika formula. Recently, Brubaker, Buciumas and Bump  showed that this matrix is equal to a certain twisted R-matrix, see \cite{BBB}.  The matrices in \cite{KP} and \cite{Mc2} were computed using an explicit realization of ${\rm I}\bigl(\sigma_s \bigr)$ which is not available if $n$ is not  prime to the residual characteristic of ${\rm F}$. It is also unclear how to extend the computations in \cite{KP} and \cite{Mc2}  to ramified representations. We take a different approach to this computation originating from the work of Arit\"urk, \cite{Ariturk},  who studied unramifed genuine principal series representations of $\widetilde{{\rm SL}_{2}({\rm F})}^{(3)}$. Moreover, the conjugacy classes of the matrices introduced in this paper are not identical to the  conjugacy classes of the matrices in \cite{Mc2}. We had to modify the construction in \cite{KP} and \cite{Mc2} to ensure that the conjugacy class of an Slcm depends only on the isomorphism class of $\sigma$ and on $\psi$, see Remark \ref{ichange} for the precise list of modifications.

In the $n=1$ case, the matrix presented here is the reciprocal of the Shahidi local coefficient defined in  \cite{Sha1}, namely $\gamma(1-s,\chi^{-1},\psi)$, where $\gamma(s,\chi,\psi)$ is the Tate $\gamma$-factor. As we explain below, once constructed in a natural way, the entries of this matrix for $n>1$ are distinguished linear combinations of $\gamma$-factors. Thus, the results in this paper give new and simple interpretation to an object otherwise known to be complicated and mysterious. For example, as we demonstrate in Section \ref{uram examples}, the fact that most of the entries of the scattering matrix appearing in  \cite{Mc2} vanish is explained by a simple property of $\epsilon$-factors. Moreover, we  show  that the  space spanned by the entries  of an Slcm associated with $\sigma$ and $\psi$ is always contained in the space of rational functions in $q^{-s}$, where $q$ is the cardinality of the residue field. In fact it equals
$$\operatorname{span} \{ s\mapsto \gamma (1-s,\chi^{-1}\eta,\psi) \mid \, \eta \in\widehat{ {\rm F^*}/{{\rm F^*}}^n} \}$$
where $\widehat{ {\rm F^*}/{{\rm F^*}}^n}$ is the group of characters of ${\rm F^*}$ of order dividing $n$.

We now give some details regarding the computation of the Slcms. Recall that the $n^{th}$ power Hilbert symbol
$$( \cdot, \cdot)_n:{\rm F^*} \times {\rm F^*} \rightarrow \mu_n$$
plays a fundamental role in the construction of $\mslt$. It defines a non-degenerate anti-symmetric bilinear form on ${\rm F^*}/{{\rm F^*}}^n$. A  main tool in this paper is the existence of a Lagrangian decomposition $L=(\overline{J}, \overline{K})$ of ${\rm F^*}/{{\rm F^*}}^n$: maximal isotropic subgroups $\overline{J}$ and $\overline{K}$ of ${\rm F^*}/{{\rm F^*}}^n$ such that $\overline{J} \times \overline{K}={\rm F^*}/{{\rm F^*}}^n$ and such that the Hilbert symbol identifies $\overline{K}$ with the dual group of $\overline{J}$. For a realization of ${\rm I}(\sigma_s)$ associated with $\overline{J}$ and a basis for $V^*$ associated with $\overline{K}$ , an Slcm, $\tau_{_L}(\cdot,\cdot,\chi,s,\psi)$ is a map defined on $\overline{K} \times \overline{K}$. In Theorem \ref{metasha} we show that
$$\tau_{_L}(a,b,\chi,s,\psi)=\gamma_{_J}\bigl(1-s, \chi^{-1}\eta_{ab},\psi,ab^{-1} \bigr),$$
where for odd $n$, $\eta_{ab}$ is the character of ${\rm F^*}$ defined by $x\mapsto (x,ab)_n$ and where $\gamma_{_J}(s,\chi,\psi,k)$ is the partial $\gamma$-factor we define in Section \ref{T and PT} by
$$\gamma_{_J}(s,\chi,\psi,k)= {(\# \overline{J})}^{-1} \sum_{j \in \overline{J}} \gamma \bigl(s,\chi\eta_j,\psi\bigr)\eta_j(k^{-1})=\int_{Jk^{-1}} \chi^{-1}(y)\ab y \ab^{1-s} \psi(y)\, d_\psi^*y.$$
Here $J$ is the pullback of $\overline{J}$ to ${\rm F^*}$ and $k \in \overline{K}$. Recall that Tate $\gamma$-factor arises from a certain functional equation  associated with a linear map defined on a one dimensional subspace of ${{\rm S}({\rm F})}^*$, the space of linear functionals on the space of Schwartz functions. In Section \ref{T and PT},  we show  that the partial $\gamma$-factors arise from a certain functional equation associated with a linear map defined on an $n$-dimensional subspace of ${{\rm S}({\rm F})}^*$. Moreover, the definition of the partial $\gamma$-factors involves  partial $\zeta$-integrals generalizing the $\zeta$-integrals which appear in  the definition of Tate $\gamma$-factor.

As an application of our approach we give in  Theorem \ref{new plan} a new formula for $\mu_n(\sigma,s)$, the Plancherel measure associated with ${\rm I}(\sigma_s)$. This  meromorphic invariant of $\sigma$ is defined by
$${\rm A}_{w^{-1}}\bigl((\sigma_s)^w\bigr )\circ {\rm A}_{w}\bigl(\sigma_s \bigr)=\mu_n(\sigma,s)^{-1}Id.$$
In the linear case, the Plancherel measure is a product of two Shahidi local coefficients, see Corollary 5.3.1 in \cite{Shabook}. In this paper we generalize this result and relate $T(\sigma,s,\psi)$ and $D(\sigma,s,\psi)$ to $\mu_n(\sigma,s)$. Moreover, utilizing our formula for the matrix of local coefficients we prove that $\mu_n(\sigma,s)$ is the harmonic mean of Plancherel measures of principal series representations of the linear group $\sltt$, namely we prove that

$$\mu_n(\sigma,s)^{-1}=[{\rm F^*}:{{\rm F^*}}^n]^{-1}\sum_{x \in {\rm F^*}/ {{\rm F^*}}^n} \mu_{1} \bigl(\chi\eta_x,s \bigr)^{-1}.$$
In our proof we use different realizations of the principal series representation simultaneously, see \eqref{aver}. We use this identity along with a global-local argument of Gao, \cite{Gao17}, and show that up to an explicit positive constant, the Plancherel measure is a quotient of $L$-functions. A similar result for parabolic induction on quasi-split reductive groups was conjectured by Langlands, \cite{Langla}, and proven for generic inducing data by Shahidi in \cite{Sha90}. We also use this formula  to find all reducible genuine principal series representations of $\mslt$ induced from a unitary data and to establish a relation between the Plancherel measures of matching representations of $\mslt^{(n)}$ and $\mslt^{(m)}$ provided that $m$ divides $n$.

Our results for $n\equiv 2 \, (\operatorname{mod }4)$ are similar. In this case the $\gamma$-factor is replaced by the metaplectic $\widetilde{\gamma}$-factor defined by us in \cite{Sz3} and $\sltt$ is replaced by $\mslt^{(2)}$. The role played by  $\widetilde{\gamma}$ in the $n\equiv 2 \, (\operatorname{mod }4)$ case is as fundamental and important as the role played by Tate $\gamma$-factor in the odd case. We have shown that no reducibilities occur on the unitary axis when $n\equiv 2 \, (\operatorname{mod }4)$. This distinction between even and odd fold covers of $\sltt$ is compatible with the fact that the dual group of $\mslt^{(n)}$ is ${\rm PGL}_2 $ if $n$ is odd and ${\rm SL} _2 $ if $n$ is even, see Page 727 of \cite{FL}.

In \cite{GanGao} and \cite{Weissman16},  Gan-Gao and Weissman raised the question whether the Langlands-Shahidi method can be extended to metaplectic groups other then the double cover of $\spn$. We believe that our results given in this paper lay the foundation for the local portion of this theory. Using the results of this paper,  a new invariant was defined in \cite{GSS}:
$$S_w(\sigma,s,\psi)=D_w(\sigma,s,\psi) \cdot \mu_n(\sigma,s)^ {e(n,F)}.$$
Here $e(n,f)$ is an explicit constant depending on $n$ and ${\rm F}$. Moreover, an evidence was given suggesting that $S_w(\sigma,s,\psi)$ is the $\gamma$-factor associated with a genuine principal series representation of $\msl$.

The paper is organized as follows: after some preparations, we define in Section \ref{funequa} the partial $\gamma$ and $\widetilde{\gamma}$-factors. In Section \ref{staninfo} we collect some information regarding the genuine principal series representations and their Whittaker functionals. Section \ref{meta sha} is devoted to the definition and study of the Slcms. Finally, in Section \ref{planme} we discuss the Plancherel measure. We finish Section \ref{planme} with a remark on the $n\equiv 0 \, (\operatorname{mod }4)$ case.
\section{General notation}

Let ${\rm F}$ be a finite extension of $\Q_p$. Denote by $q$ the cardinality of its residue field. Denote by $\Of$ its ring of integers. Fix $\varpi$, a generator of $\Pf$, the maximal ideal of $\Of$. We normalize the absolute value on ${\rm F}$ such that $\ab \varpi \ab =q^{-1}$.

Let $\psi$ be a non-trivial character of ${\rm F}$. We shall denote by $d_\psi x$ the Haar measure on ${\rm F}$ which is self dual with respect to $\psi$ and we set $d^*_\psi x=\frac {d_\psi x}{\ab x \ab}.$ It is a Haar measure on ${\rm F^*}$. For $a \in {\rm F^*}$ let $\psi_a$ be the character of ${\rm F}$ given by $x \mapsto \psi(ax)$. We define $e(\psi)$, the conductor of $\psi$, to be the smallest positive integer $k$ such that $\psi$ is trivial on $\Pf^k$. We say that $\psi$ is normalized if $e(\psi)=0$. For a ramified character $\chi$ of ${\rm F^*}$ we define $e(\chi)$, the conductor of $\chi$, to be the smallest integer $k$ such that $\chi$ is trivial on $1+\Pf^k$. For an unramified character $\chi$ of ${\rm F^*}$ we set $e(\chi)=0$.

Last, for any group $T$ and a character $\alpha$ of $T$ we denote by $\C_\alpha$ the one dimensional complex space on which $T$ acts by $\alpha$.
\section{Functional equations} \label{funequa}
\subsection{ A Lagrangian decomposition} \label{npower}
 Fix an integer $n \geq 1$  (starting at Section  \ref{modelsection} we shall assume that $n$ is not divisible by 4). We shall assume that ${\rm F}^*$ contains the full group of $n^{th}$ roots of 1. Denote this cyclic group by $\mu_n$. We identify $\mu_n$ with the group of $n^{th}$ roots of 1 in $\C^*$ and suppress this identification.

For  $m\in \N$ which divides $n$ let $$( \cdot, \cdot)_{m}:{\rm F}^* \times {\rm F}^* \rightarrow \mu_{m}$$ be the $m^{th}$ power Hilbert symbol. Recall that the Hilbert symbol is an anti-symmetric bilinear form and that its kernel in each argument is ${{\rm F}^*}^{m}$. Hence, it gives rise to a non-degenerate bilinear form on ${\rm F}^* / {{\rm F}^*}^{m} \times {\rm F}^* / {{\rm F}^*}^{m}$. In particular, it identifies ${\rm F}^* / {{\rm F}^*}^{m}$ with its dual, $\widehat{{\rm F}^* / {{\rm F}^*}^{m}}$, which may also  be identified with the group of  characters of ${\rm F}^*$ whose order divides $m$. Note that since ${\rm F}^*$ contains $\mu_{m}$
\begin{equation} \label{index} [{\rm F}^*:{{\rm F}^*}^{m}]=m^2 \ab m \ab^{-1}, \end{equation} see Page 48 of \cite {Lang} for example.
For all   $x \in {\rm F}^*$, \begin{equation} \label{x with x} (x,x)_{m}=(-1,x)_{m}. \end{equation}
Also, if $n=m l$ then
\begin{equation} \label{FV fact} (x,y)_n^{m}=(x,y)_{l}. \end{equation}
See Section 5 in Chapter IV of \cite{FV} for example. Denote $$d=\begin{cases} n &  n \, \operatorname{is} \,  \operatorname{odd }; \\  \frac {n} {2}  &  n \, \operatorname{is} \,  \operatorname{even. }  \end{cases}$$
Observe that ${\rm F}^*$ contains $\mu_{2d}$. Indeed, if $n$ is even there is nothing to explain while if $n$ is odd we note that
$\mu_n \cup -\mu_n=\mu_{2n}.$ Thus, $-1 \in {{\rm F}^*}^d.$ Using \eqref{x with x} we conclude that for all $x \in {\rm F}^*$
\begin{equation} \label{x with x is 1}(x,x)_d=1. \end{equation}
For $x \in {\rm F}^*$ let $\eta_x$ be the character of ${\rm F}^*$ defined by
$$y \mapsto \eta_x(y)=(x,y)_d.$$
Note that the map $x \mapsto \eta_x$ factors through ${\rm F}^*/{{\rm F}^*}^d$. When convenient we  shall also think of $\eta_x$ as an element in  $\widehat{{\rm F}^*/{{\rm F}^*}^d}$. \begin{deff} A subgroup $\overline{J}$ of ${\rm F}^*/{{\rm F}^*}^d$ (of ${\rm F}^*$) is called a Lagrangian subgroup if
$\overline{J} =\bigcap_{x \in {\overline{J}}} \operatorname{ker}(\eta_x).$ \end{deff}
\begin{deff}Let $\overline{J}$ and $\overline{K}$ be two Lagrangian subgroups of ${\rm F}^*/{{\rm F}^*}^d$. We say that $L=(\overline{J},\overline{K})$ is a Lagrangian decomposition  of ${\rm F}^*/{{\rm F}^*}^d$ if ${\rm F}^*/{{\rm F}^*}^d$ is a direct product of $\overline{J}$ and $\overline{K}$ and the map $k \mapsto {\eta_k \! \mid}_{\overline{J}}$ is an isomorphism from $\overline{K}$ to the dual group of $\overline{J}$ (in particular $\overline{J} \cong \overline{K}$). \end{deff}
Note that if $(\overline{J},\overline{K})$  is a Lagrangian decomposition of ${\rm F}^*/{{\rm F}^*}^d$ then by \eqref{index} we have
\begin{equation} \nonumber \# \overline{J}=\# \overline{K} =\sqrt{[{\rm F}^*:{{\rm F}^*}^d]}=d \ab d \ab^{-\half}. \end{equation}
\begin{lem} \label{lang decomp} A Lagrangian decomposition of ${\rm F}^*/{{\rm F}^*}^d$ exists. Furthermore, if $(\overline{J},\overline{K})$ and $(\overline{J}',\overline{K}')$ are two Lagrangian decompositions  of  ${\rm F}^*/{{\rm F}^*}^d$ then there exists an automorphism $\theta$ of ${\rm F}^*/{{\rm F}^*}^d$ preserving $(\cdot, \cdot)_d$ such that $\theta(\overline{J})=\overline{J}'$ and $\theta(\overline{K})=\overline{K}'$.
\end{lem}
\begin{proof}  We work in a more general setting studied by Davydov in  \cite{dav}. Let $A$ be a finite abelian group and let $k$ be a field. Davydov defines a non-degenerate bilinear form
$$[\cdot, \cdot]:A \times A \rightarrow k^*$$ to be alternative if $[x,x]=1$ for all  $x \in A$ and proves in Lemma 4.2 of \cite{dav} that a Lagrangian decomposition $A=J \times K$ exists. By \eqref{x with x is 1}, this proof applies to our case. Let ${\rm Sp}(A)$ be the the group of automorphisms of $A$ preserving $[\cdot, \cdot]$. It remains to prove
that ${\rm Sp}(A)$ acts transitively on the set of Lagrangian decompositions. So consider a second Lagrangian decomposition $A=J' \times K'$. By the structure
theorem for finite abelian groups, $K \cong \operatorname{Hom}(J,\C^*) \cong J$ (non-canonically) and
similarly $K' \cong J'$. Thus $A \cong J \times J \cong J' \times J'$. Again by the structure theorem for
finite abelian groups, it follows that $J \cong J'$. Let $\alpha:J \rightarrow J'$ be an isomorphism and let  $\alpha^*:K \rightarrow K'$ be the dual isomorphism, i.e., $[\alpha^*(k), \alpha(j)] = [k, j]$ for all $j \in J$ and $k \in K$. By construction $\theta=\alpha \times \alpha^* \in {\rm Sp}(A)$ and $\theta$ sends $(J,K)$ to $(J',K')$.
\end{proof}
\begin{example} \label{kp example} In the case where ${\rm gcd}(d,p)=1$ set
$$\overline{J}=\Of^* {{\rm F}^*}^d / {{\rm F}^*}^d,\,  \overline{K}=<\omega> {{\rm F}^*}^d / {{\rm F}^*}^d.$$
It follows from Chap. XIII, Sec. 5 of \cite{Weilbook} that $L=(\overline{J}, \overline{K})$ is a Lagrangian decomposition of  ${{\rm F}^*} / {{\rm F}^*}^d$.
Both $\overline{J}$ and $\overline{K}$ are isomorphic to the cyclic group with $d$ elements.
\end{example}
\begin{remark} Suppose that ${\rm gcd}(d,p)=1$ and suppose in addition that  $d=m^2$ for some integer $m>1$. Then $A={{\rm F}^*}^m / {{\rm F}^*}^d$ is a Lagrangian subgroup. But since $A$ is not cyclic it follows from  Lemma \ref{lang decomp} combined with Example \ref{kp example} that $A$ does not fit into a Lagrangian decomposition.  In example  4.4 of \cite{dav} the author shows directly that the extension $$A \rightarrow {{\rm F}^*} / {{\rm F}^*}^d \rightarrow ({{\rm F}^*} / {{\rm F}^*}^d)/A$$ does not split. This shows that, contrary to the symplectic context, not every Lagrangian subgroup produces a Lagrangian decomposition.
\end{remark}

We shall now fix once and for all a Lagrangian decomposition $L=(\overline{J},\overline{K})$ of ${{\rm F}^*} / {{\rm F}^*}^d$. Let $J$ and $K$ be the pullbacks of $\overline{J}$ and $\overline{K}$ respectively  to ${\rm F}^*$. Both $J$ and $K$ are Lagrangian subgroups of ${\rm F}^*$, $JK={\rm F}^*$ and $J \cap K={{\rm F}^*}^d.$
\begin{remark} \label{nicedid}  We note that $K \cong {\rm F}^*/J$. Indeed, the natural map, inclusion into ${\rm F}^*/{{\rm F}^*}^d$ and then projection modulo $\overline{J}$ is an isomorphism from $\overline{K}$ to ${\rm F}^*/J$. Since the Hilbert symbol $(\cdot,\cdot)_d: {\rm F}^*/{{\rm F}^*}^d \times {\rm F}^*/{{\rm F}^*}^d \rightarrow \mu_d$  restricts to a perfect pairing $\overline{K}  \times \overline{J} \rightarrow \mu_d$ the isomorphism above gives a perfect pairing ${\rm F}^*/J  \times \overline{J} \rightarrow \mu_d$. Moreover, while
$(\cdot,\cdot)_d$ is not well defined on ${\rm F}^*/J \times {\rm F}^*/J $ , it is well defined and trivial on $\overline{K} \times \overline{K}.$ Thus, we always choose representatives for  ${\rm F}^*/J$ in $K$.
\end{remark}
Remark \ref{nicedid} along with the orthogonality of characters of finite groups implies that  for all $x, k \in {\rm F}^*$,  \begin{equation} \label{char func}{(\# \overline{J})}^{-1}\sum_{j \in {\overline{J}} }\eta_j\bigl(xk^{-1}\bigr)=\begin{cases} 1 &  \operatorname{if }x \in Jk \\  0  &  \operatorname{othetwise.}  \end{cases}. \end{equation}
\begin{lem} \label{res hil} Fix $x,y \in J$.
$$(x,y)_n=\begin{cases} 1 &  n \, \operatorname{is} \,  \operatorname{odd }; \\  (x,y)_{2}  &  n\equiv 2 \, (\operatorname{mod }4). \end{cases}$$
\end{lem}
\begin{proof}  This lemma is trivial if $n$ is odd. If $n \equiv 2 \, (\operatorname{mod }4)$ then $d$ is odd. Also $(x,y)_n =\pm 1$ for $x,y \in J$ since by \eqref{FV fact}, $(x,y)^2_n=(x,y)_d=1$. These two observations along with another application of \eqref{FV fact} imply that for all $x,y \in J$
$$(x,y)_n=(x,y)_n^{d}=(x,y)_{_2}.$$
\end{proof}
\subsection{Eigen functionals} \label{eigenfun}
Let $\chi''$ be a character of ${{\rm F}^*}^d$. Since $[{\rm F}^*:{{\rm F}^*}^d]<\infty$ it follows as in the finite abelian group setting that $\chi''$ can be extended to ${\rm F}^*$. Let $\chi$ be one these extensions and let  $\chi'$ be the restriction of $\chi$ to $J$. This notation will be fixed throughout this paper. Note that
$$\{\chi\eta_y \mid y \in {\rm F}^*/{{\rm F}^*}^d \}$$ is the set of all extensions of $\chi''$ to ${\rm F}^*$,
$$\{\chi'\eta_y \mid y \in \overline{K} \}$$ is the set of all extensions of $\chi''$ to $J$ and
$$\{\chi\eta_y \mid y \in {\overline{J}} \}$$ is the set of all extensions of $\chi'$ to ${\rm F}^*$.

Let ${\rm S}({\rm F})$ be the space of Schwartz functions on ${\rm F}$ and let ${{\rm S}({\rm F})}^*$ be the space of linear functionals on ${\rm S}({\rm F})$. ${\rm F}^*$ acts on  ${\rm S}({\rm F})$ by right translations. Denote this action by $\rho$. ${\rm F}^*$ also acts on ${{\rm S}({\rm F})}^*$ by
$$\bigl(\lambda(g)\xi \bigr)\phi=\xi \bigl(\rho(g^{-1}) \phi \bigr).$$
Let ${{\rm S}({\rm F})}^*_{\chi}$  $\bigl( {{\rm S}({\rm F})}^*_{J, \chi'}\bigr)$  be the space of $\chi$ $(\chi')$ eigenfunctionals on ${\rm S}({\rm F})$.
\begin{lem} \label{deeptate} ${{\rm S}({\rm F})}^*_{\chi}$ is a one dimensional space.
\end{lem}
This fundamental uniqueness result is implicit in Tate's thesis, \cite{T}, and is proven in \cite{Weil66}. See also  Theorem 3.4 of \cite{Kudla}.
\begin{lem} \label{eigen space}
$${{\rm S}({\rm F})}^*_{J, \chi'}=\bigoplus_{j \in \overline{J}} {{\rm S}({\rm F})}^*_{\chi \eta_j}.$$
In particular, $\dim \,{{\rm S}({\rm F})}^*_{J, \chi'}=d \ab d \ab^{-\half}.$
\end{lem}
\begin{proof}
Using Frobenius reciprocity we obtain
$${{\rm S}({\rm F})}^*_{J, \chi'}=\operatorname{Hom}_J \bigl({\rm S}({\rm F}),\C_{\chi'} \bigr) \simeq \operatorname{Hom}_{{\rm F}^*} \bigl({\rm S}({\rm F}),{\rm Ind}_J^{{\rm F}^*}\C_{\chi'} \bigr).$$
Since ${\rm Ind}_J^{{\rm F}^*}\C_{\chi'}=\bigoplus_{j \in \overline{J}} \C_{\chi\eta_j}$ the first assertion follows. The second assertion now follows form Lemma \ref{deeptate}.
\end{proof}
For $s\in \C$ let $\chi_s$ be the character of ${\rm F}^*$ given by
$$x \mapsto \chi(x) \ab x \ab^s.$$
The same definition applies for characters of subgroups of ${\rm F}^*$. It was shown in Tate's Thesis, \cite{T}, that a non-zero element in  ${{\rm S}({\rm F})}^*_{\chi_s}$ is given by
$$\phi \mapsto {L(s,\chi)}^{-1}\zeta(s,\chi,\phi)$$ where $$L(s,\chi)=\begin{cases}  \frac {1} {1-q^{-s}\chi(\varpi)}\ & \chi \, \operatorname{is} \,  \operatorname{unramified} ;\\  1  &  \operatorname{otherwise} \end{cases}$$ and where $\zeta(s,\chi,\phi)$ is the rational function in $q^{-s}$ given by the meromorphic continuation of
$$\int_{{\rm F}^*} \phi(x) \chi_s(x) \, d_\psi^*x.$$
This integral is absolutely convergent for $Re(s) \gg 0$.  For $k \in \overline{K}$ we set
\begin{equation} \label{parz} \zeta_J(s,\chi,\phi,k)={(\# \overline{J})}^{-1}\sum_{j \in {\overline{J}} }\eta_j(k^{-1})\zeta(s,\chi\eta_j,\phi).\end{equation}
For $j \in  \overline{J}, k \in  \overline{K}$ we define away from the poles $\alpha_{j,\chi,s}\,  \beta_{k,\chi,s} \in {{\rm S}({\rm F})}^*$ by
\begin{equation} \label{funcz}\alpha_{j,\chi,s}(\phi)=\zeta(s,\chi\eta_j,\phi), \, \, \beta_{k,\chi,s}(\phi)=\zeta_k(s,\chi,\phi,k) \end{equation}
and we set $$\mathcal{A}_{\chi,s}=\{\alpha_{j,\chi,s} \mid \j \in \overline{J} \}, \, \, \mathcal{B}_{\chi,s}=\{\beta_{k,\chi,s} \mid \j \in \overline{J} \}.$$
From Lemma \ref{eigen space} it follows that $\mathcal{A}_{\chi,s}$ is a basis for  ${{\rm S}({\rm F})}^*_{J, \chi'_s}$.
\begin{lem} \label{partial zeta} Away from the poles, $\mathcal{B}_{\chi,s}$ is a basis for ${{\rm S}({\rm F})}^*_{J, \chi_s'}$. Also, $\zeta_J(s,\chi,\phi,k)$ is the meromorphic continuation of
$$\int_{Jk} \phi(x) \chi_s(x) d_\psi^*x.$$
This integral converges absolutely for $Re(s) \gg 0$.
\end{lem}
\begin{proof} From the duality of $\overline{J}$ and $\overline{K}$ it follows that
\begin{equation}\label{invparz} \zeta(s,\chi\eta_j,\phi)=\sum_{k \in \overline{K}} \eta_j(k)\zeta_J(s,\chi,\phi,k).\end{equation}
Hence, $\mathcal{A}_{\chi,s}\subseteq  \operatorname{span} \mathcal{B}_{\chi,s}$. The first assertion is  proven since
$\# \mathcal{B}_{\chi,s} =\dim  \, {{\rm S}({\rm F})}^*_{J, \chi'}.$

We now prove the second assertion.  Suppose that $Re(s) \gg 0 $ so that all the integrals below are absolutely convergent. We have
\begin{eqnarray} \nonumber
\zeta_J(s,\chi,\phi,k)&=&{(\# \overline{J})}^{-1}\sum_{j \in \overline{J}}\eta_j(k^{-1})\zeta(s,\chi\eta_j,\phi)={(\# \overline{J})}^{-1}\sum_{j \in \overline{J}}\int_{{\rm F}^*} \phi(x) \chi_s(x)\eta_j(xk^{-1}) d_\psi^*x
\\ \nonumber &=&\int_{{\rm F}^*} \Bigl({(\# \overline{J})}^{-1}\sum_{j \in \overline{J}}\eta_j(xk^{-1}) \Bigr) \phi(x) \chi_s(x) d_\psi^*x.
\end{eqnarray}
By \eqref{char func} the proof is done.
\end{proof}
The following lemma will not be used later. It is included here for the sake of completeness.
\begin{lem} For $k \in {\rm F}^*$ let ${{\rm S}({\rm F})}^*_{J, \chi',k}$ be the subspace of ${{\rm S}({\rm F})}^*_{J, \chi'}$ which consists of functionals supported on $Jk$. Then $\zeta_J(s,\chi,\phi,k) \in {{\rm S}({\rm F})}^*_{J, \chi_s',k}$ and
$\dim \,  {{\rm S}({\rm F})}^*_{J, \chi',k} =1.$
\end{lem}
\begin{proof}The first assertion follows immediately from the integral formula given in Lemma \ref{partial zeta}. In particular, for any $k \in {\rm F}^*$, $\dim \, {{\rm S}({\rm F})}^*_{J, \chi',k}  \geq 1.$
The second assertion now follows since $$\bigoplus_{k \in \overline{K}} {{\rm S}({\rm F})}^*_{J, \chi',k} \subseteq {{\rm S}({\rm F})}^*_{J, \chi'}$$
and $\#  \overline{K}=\dim \, {{\rm S}({\rm F})}^*_{J, \chi'}.$
\end{proof}
\subsection{Tate $\gamma$-factor} \label{T and PT}
We fix a non-trivial character $\psi$ of ${\rm F}$. Given $\phi \in {\rm S}({\rm F})$ we denote by $\widehat{\phi} \in {\rm S}({\rm F})$  its $\psi$-Fourier transform, i.e.,
$$\widehat{\phi}(x)=\int_F \phi(y) \psi(xy) \, d_\psi y.$$
For $j \in  \overline{J}$ we define away from the poles $\widehat{\alpha}_{j,\chi,s} \in {{\rm S}({\rm F})}^*$ by
$$\widehat{\alpha}_{j,\chi,s}(\phi)=\alpha_{j^{-1},\chi^{-1},1-s}(\widehat{\phi})$$
By Tate's thesis, \cite{T}, for all $s \in \C$
$${L(1-s,\chi^{-1})}^{-1}\widehat{\alpha}_{j,\chi,s} \in {{\rm S}({\rm F})}^*_{\chi_s}.$$ This observation along with the uniqueness result in Lemma \ref{deeptate} give rise to Tate $\gamma$-factor
$$\gamma(s,\chi,\psi)=\epsilon(s,\chi,\psi)\frac{L(1-s,\chi^{-1})}{L(s,\chi)}$$
This rational function in $q^{-s}$ is defined via the functional equation.
\begin{equation} \label{tate def} \zeta(1-s,\chi^{-1},\widehat{\phi})=\gamma(s,\chi,\psi)\zeta(s,\chi,\phi). \end{equation}
By plugging to \eqref{tate def} a test function $\phi$ supported in $1+\Pf^r$ where $r \gg 0$ one shows that ${\gamma}(1-s,\chi^{-1}, \psi)$ is given by the  meromorphic continuation of
\begin{equation} \label{tateintegral} \lim_{r \rightarrow \infty}\int_{\Pf^{-r}}\chi_{s}(x)\psi(x) \, d_\psi^* x. \end{equation}
This limit exists for $Re(s) \gg 0$.
It is well known that $\epsilon(s,\chi,\psi)$ is a monomial function in $q^{-s}$ and that $\epsilon(s,\chi,\psi)$=1 if $\chi$ is unramified and $\psi$ is normalized.
 Also,
\begin{eqnarray} \label{Tate gamma} {\epsilon}(1-s,\chi^{-1},\psi) &=& \chi(-1)\epsilon(s,\chi,\psi)^{-1}, \\
\label{changepsi}\gamma(s,\chi,\psi_a) &=& \chi(a)\ab a \ab^{s-\half}\gamma(s,\chi,\psi), \\
\label{epsilon old twist} \epsilon(s+t,\chi,\psi) &=& q^{e(\psi)-e(\chi)t}\epsilon(s,\chi,\psi),\\
\label{epsilon twist} \epsilon(s,\chi\eta,\psi) &=& \eta(\varpi)^{e(\chi)-e(\psi)}\epsilon(s,\chi,\psi),\end{eqnarray}
where $\eta$ is an unramified character. See Section 1 in \cite{Schmidt} for example. Combining \eqref{Tate gamma} and \eqref{epsilon old twist} one obtains
\begin{equation} \label{epsilon twist and inv}
\epsilon(1-s,\chi^{-1},\psi)\epsilon(1+s,\chi,\psi)= \chi(-1)q^{e(\psi)-e(\chi)}.
\end{equation}
\subsection{Partial $\gamma$-factors} \label{T and PT}
We define a linear operator ${\rm M}_\psi$ on ${{\rm S}({\rm F})}^*_{J, \chi'}$ by setting
$${\rm M}_\psi \bigl({\alpha}_{j,\chi,s}\bigr)=\widehat{\alpha}_{j,\chi,s}$$
and then by linear extension. For $k \in \overline{K}$ denote
\begin{equation} \label{motivate} {\rm M}_\psi \bigl({\beta}_{k,\chi,s}\bigr)=\widehat{\beta}_{k,\chi,s} \end{equation}
Recalling \eqref{parz} and \eqref{funcz} one observes that
$$\widehat{\beta}_{k,\chi,s}(\phi)=\beta_{k^{-1},\chi^{-1},1-s}(\widehat{\phi}).$$
By  \eqref{tate def} the matrix representing ${\rm M}_\psi$ with respect to $\mathcal{A}_{\chi,s}$ is diagonal. We now describe the matrix representing ${\rm M}_\psi$ with respect to $\mathcal{B}_{\chi,s}$.
For $k \in \overline{K}$ define the partial $\gamma$-factor
\begin{equation}\label{part gamma def} \gamma_{_J}(s,\chi,\psi,k)={(\# \overline{J})}^{-1} \sum_{j \in \overline{J}} \gamma(s,\chi\eta_j,\psi)\eta_j(k^{-1}).\end{equation}
It is a rational function in $q^{-s}$.
\begin{thm}\label{partial gamma lemma}
$$\zeta_{_J}(1-s,\chi^{-1},\widehat{\phi},k_0^{-1})=\sum_{k \in \overline{K}}{ \gamma}_{_J}(s,\chi,\psi,k^{-1}k_0) \zeta_{_J}(s,\chi,\phi,k).$$

Also, $ \gamma_{_J}(1-s,\chi^{-1},\psi,k)$ is the meromorphic continuation of
$$\lim_{r \rightarrow \infty} \int_{\Pf^{-r} \cap Jk^{-1}}  \chi_{s}(x) \psi(x)\, d_\psi^*x.$$
\end{thm}
\begin{proof}
\begin{eqnarray} \nonumber
\bigl({\rm M}_\psi ( {\beta}_{k_0,\chi,s}) \bigr)(\phi) &=&{(\# \overline{J})}^{-1} \sum_{j \in \overline{J}} \eta_j(k_0^{-1})\bigl({\rm M}_\psi  ({\alpha}_{k_0,\chi,s}) \bigr)(\phi)  \\ \nonumber
&=&{(\# \overline{J})}^{-1} \sum_{j \in \overline{J}} \eta_j(k_0^{-1}) \gamma(s,\chi\eta_j,\psi)\zeta(s,\chi\eta_j,\phi).
\end{eqnarray}
By \eqref{invparz}  we obtain
\begin{eqnarray} \nonumber
\bigl({\rm M}_\psi ( {\beta}_{k_0,\chi,s}) \bigr)(\phi)&=&{(\# \overline{J})}^{-1} \sum_{j \in \overline{J}} \eta_j(k_0^{-1}) \gamma(s,\chi\eta_j,\psi)\sum_{k \in  \overline{K}} \eta_j(k)\zeta_{_J}(s,\chi,\phi,k)\\ \nonumber
&=&\sum_{k \in \overline{K}} \Bigl( {(\# \overline{J})}^{-1} \sum_{j \in \overline{J}} \gamma(s,\chi\eta_j,\psi)\eta_j(kk_0^{-1}) \Bigr)\zeta_{_J}(s,\chi,\phi,k).
\end{eqnarray}
This proves the first assertion.

We now prove the second assertion. We may assume that $Re(s) \gg 0$ so that all the limits below exist.
\begin{eqnarray} \nonumber  \gamma_{_J}(1-s,\chi^{-1},\psi,k)&=&{(\# \overline{J})}^{-1}\sum_{j \in \overline{J}} \eta_j(k^{-1})\gamma(1-s,\chi^{-1}\eta_j,\psi)\\ \nonumber &=&{(\# \overline{J})}^{-1}\sum_{j \in \overline{J}} \eta_j(k^{-1}) \lim_{r \rightarrow \infty} \int_{\Pf^{-r}} \psi(x)(\chi \eta^{-1}_j)_{s}(x) d_\psi^*x\\ \nonumber
 &=& \lim_{r \rightarrow \infty} \int_{\Pf^{-r}} \Bigl({(\# \overline{J})}^{-1} \sum_{j \in \overline{J}} \eta^{-1}_j(kx) \Bigr) \psi(x) \chi_s(x) d_\psi^*x.
 \end{eqnarray}
Using  \eqref{char func} again we are done.
\end{proof}
\begin{remark} \nonumber We have been using the notation $\zeta_J(s,\chi,\phi,k)$ and $\gamma_{_J}(s,\chi,\psi,k)$ rather then $\zeta_L(s,\chi,\phi,k)$ and $\gamma_{_L}(s,\chi,\psi,k)$  since the results in Sections \ref{eigenfun}, \ref{T and PT} and in Section \ref{MT and PMT} below do not depend on the existence of a Lagrangian decomposition. These results only uses the fact that $J$ is a maximal abelian subgroup. $\overline{K}$ can be replaced by $F/J$.
\end{remark}
\subsection{Weil index} \label{weilindexdef}
Let $\gamma_F(\psi)$ be the unnormlized Weil index  defined and studied in \cite{Weil}. It is given by
$$\gamma_F(\psi)=\lim_{r \rightarrow \infty}\int_{\Pf^{-r}}\psi(x^2)\, d_{\psi_{_2}} x.$$
In fact $$\int_{\Pf^{-r}}\psi(x^2)\, d_{\psi_{_2}} x$$ stabilizes for  $r \gg 0$ and
$\mid \! \gamma_F(\psi) \! \mid=1$, see Propositions 3.1 and 3.3 in \cite{CC}.
In particular, since $\overline{\psi}=\psi_{-1}$ we have
\begin{equation} \label{weilchangepsii} \gamma_F(\psi)^{-1}=\overline{\gamma_F(\psi)}=\gamma_F(\psi_{-1}). \end{equation}
For $a \in {\rm F}^*$ define now the normalized Weil index
\begin{equation} \label{norweildef} \gamma_\psi(a)=\frac{\gamma_F(\psi_a)}{\gamma_F(\psi)}.\end{equation}
By changing the integration variable one immediately observes that $\gamma_F(\psi_{a^2})=\gamma_F(\psi)$ for any $a \in {{\rm F}^*}^2$. Equivalently $\gamma_\psi\bigl({{\rm F}^*}^2\bigr)=1$. Note that from the definition of $\gamma_\psi(-1)$ combined with  \eqref{weilchangepsii} it follows that
\begin{equation} \label{unnor weil} \gamma_F(\psi_{-1})^{2}=\gamma_\psi(-1). \end{equation}
It is well known that for all $x,y \in {\rm F}^*$
\begin{equation} \label{weil prop} \gamma_\psi(xy)=\gamma_\psi(x)\gamma_\psi(y)(x,y)_2, \end{equation}
see for example Page 367 in \cite{Rao}. Note that \eqref{weil prop} implies that $\gamma_\psi(a)$ lies in $\mu_4$. Equation \eqref{unnor weil} now implies that $\gamma_F(\psi)$ lies in $\mu_8$.  Equation \eqref{unnor weil} also implies that
\begin{equation} \label{weilchangepsi} \gamma_{\psi_a}(x)=(a,x)_2\gamma_\psi(x). \end{equation}
In particular, since $\gamma_\psi^{-1}=\gamma_{\psi_{-1}}$ it follows that
\begin{equation} \label{weilchangepsiinv} \gamma_{\psi}(x)^{-1}=(-1,x)_2\gamma_\psi(x). \end{equation}
We note that  \eqref{norweildef} and \eqref{x with x} imply that
\begin{equation} \label{weilsqyare} \gamma_\psi(a)^2=(a,-1)_2 \end{equation}
Last, note that \eqref{unnor weil} and \eqref{weilsqyare} imply  that
\begin{equation} \label{lasttwist} \gamma_F(\psi_{-1})^{-2}\gamma_\psi(-1)^{-1}=(-1,-1)_2. \end{equation}
\subsection{Metaplectic $\widetilde{\gamma}$-factor} \label{MT and PMT}
For $\phi \in {\rm S}({\rm F})$ we  define $\widetilde{\phi}:{\rm F}^* \rightarrow \C$ by
$$\widetilde{\phi}(x)=\int_{{\rm F}^*} \phi(y)\gamma_\psi(xy)^{-1} \psi(xy) d_\psi y.$$
Although $\widetilde{\phi}(x)$ is typically not an element of ${\rm S}({\rm F})$ it was proven in \cite{Sz3} that
$$\int_{{\rm F}^*} \widetilde{\phi}(x) \chi_s(x) d^*_\psi x$$
converges absolutely for $a<Re(s)<a+1$, for some  $a \in \R$, to a rational function in $q^{-s}$. This enables the natural definition of $\zeta(s,\chi,\widetilde{\phi})$ as the meromorphic continuation of the integral above.  Furthermore, by \cite{Sz3} there exists a metaplectic $\widetilde{\gamma}$-factor,  $\widetilde{\gamma}(s,\chi,\psi)$, such that
$$\zeta(1-s,\chi^{-1},\widetilde{\phi})=\zeta(s,\chi,\phi)\widetilde{\gamma}(s,\chi,\psi)$$
for all $\phi \in {\rm S}({\rm F})$.  In particular, away from the poles,  $$\phi \mapsto \zeta(1-s,\chi^{-1},\widetilde{\phi})\in {{\rm S}({\rm F})}^*_{\chi_s}.$$  It was also proven in \cite{Sz3} that $\widetilde{\gamma}(1-s,\chi^{-1},\psi)$ is the meromorphic continuation of
$$\lim_{r \rightarrow \infty}\int_{\Pf^{-r}}\chi_s(x)  \gamma_\psi(x)^{-1} \psi(x) \, d_\psi^* x.$$
This limit exists for $Re(s) \gg 0$. The computation of the last integral is contained in an unpublished note of W. Jay Sweet, \cite{Sweet}, see also the appendix of \cite{GoSz}:
\begin{equation} \label{meta gama formula} \widetilde{\gamma}(1-s,\chi^{-1},\psi)=\gamma_F(\psi_{-1})^{-1} \chi(-1) \frac{ \gamma(s+\half,\chi,\psi)}{ \gamma(2s,\chi^{2},{\psi_{_2}})}. \end{equation}
From  \eqref{meta gama formula}, \eqref{changepsi} and \eqref{norweildef} we have
\begin{equation} \label{metachangepsi}\widetilde{\gamma}(s,\chi,\psi_a)=\gamma_\psi(a)\ab a \ab ^{s-\half} \chi(a)\widetilde{\gamma}(s,\chi,\psi).\end{equation}
\subsection{Partial $\widetilde{\gamma}$-factors} \label{MT and PMT}
Similar to Section \ref{T and PT},  for $k \in \overline{K}$ we set
\begin{equation}\label{part metagamma def}\zeta_J(s,\chi,\widetilde{\phi},k)={(\# \overline{J})}^{-1}\sum_{j \in  {\overline{J}} }\eta_j(k^{-1})\zeta(s,\chi\eta_j,\widetilde{\phi}).\end{equation}
$\zeta_J(s,\chi,\widetilde{\phi},k)$ is the meromorphic continuation of
$$\int_{Jk^{-1}} \widetilde{\phi}(x) \chi_s(x) d_\psi^*x.$$
 For $k \in \overline{K}$ we define the partial metaplectic $\widetilde{\gamma}$-factor
\begin{equation} \label{part meta  gamma def}\widetilde{\gamma}_{_J}(s,\chi,\psi,k)={(\# \overline{J})}^{-1} \sum_{j \in \overline{J}} \widetilde{\gamma}(s,\chi\eta_j,\psi)\eta_j(k^{-1}).\end{equation}
We have the following analog to Theorem \eqref{partial gamma lemma}
\begin{thm} \label{partial metagamma lemma}
$$\zeta_{_J}(1-s,\chi^{-1},\widetilde{\phi},k_0^{-1})=\sum_{k \in \overline{K}}\widetilde{ \gamma}_{_J}(s,\chi,\psi,k^{-1}k_0) \zeta_{_J}(s,\chi,\phi,k).$$
Also, ${\widetilde{\gamma}}_{_J}(1-s,\chi^{-1},\psi,k)$ is the meromorphic continuation of

$$ \lim_{r \rightarrow \infty}\int_{\Pf^{-r} \cap Jk^{-1}}\chi_s(x) \gamma_\psi(x)^{-1} \psi(x) \, d_\psi^* x.$$
This last limit exists for $Re(s) \gg 0$.
\end{thm}
\section{Genuine principal series representations of $\widetilde{\rm {SL}_2({\rm F})}$.} \label{staninfo}
\subsection{An $n$ fold cover of ${\rm SL}_2({\rm F})$.}
Let $G={{\rm SL}_2({\rm F})}$ be the group of two by two matrices with entries in ${\rm F}$ whose determinant is 1. Let ${\rm N} \cong {\rm F}$ be the group of upper triangular unipotent matrices. Let ${\rm H}\cong {\rm F}^*$ be the group of diagonal elements inside  $G$. Denote ${\rm B}={\rm H} \ltimes {\rm N}$. For $x \in { \rm F}$, and  $a\in {\rm F}^*$ we shall write
$$n(x)=\left( \begin{array}{cc} {1} & {x} \\ {0} & {1} \end{array} \right), \quad  h(a)=\left( \begin{array}{cc} {a} & {0} \\ {0} & {a^{-1}} \end{array} \right), \quad  w_{_0}=\left( \begin{array}{cc} {0} & {1} \\ {-1} & {0} \end{array} \right).$$
Let $\mslt^{(n)}=\widetilde{{\rm SL}_2({\rm F})}^{(n)}$ be the topological central extension of $\sltt$ by $\mu_n$ constructed by Kubota in \cite{Kub}. We have the short exact sequence
$$1\rightarrow \mu_n \rightarrow \mslt  \rightarrow  \sltt \rightarrow  1.$$
We shall realize  $\mslt^{(n)}$ as the set $G \times \mu_n$ along with the multiplication
$$\bigl(g,\epsilon \bigr)\bigl(g',\epsilon' \bigr)=\bigl(gg',c(g,g')\epsilon \epsilon'\bigr),$$
where \begin{equation}\label{rao}c(g,g')=\bigl(x(gg')x^{-1}(g),x(gg')x^{-1}(g')\bigr)_n.\end{equation} Here
$$x \left( \begin{array}{cc} {a} & {b} \\ {c} & {d} \end{array} \right)=\begin{cases} c & c \neq 0; \\ d & c=0.
\end{cases}$$
We shall denote by $sec$ the section map from  $\sltt$ to $\mslt^{(n)}$ given by $$sec(g)=(g,1)$$ (generally, it is not a group homomorphism). We set $w=sec(w_{_0})$. For a subset $A$ of $SL_{2}({\rm F})$ we shall denote by $\widetilde{A}^{(n)}$ its inverse image in  $\mslt^{(n)}$. For most of this paper, $n$ is fixed. Thus, when convenient we shall drop the index $n$  when discussing  $\mslt^{(n)}$ and its subgroups (only toward  the end of Section \ref{avformula} $n$ varies).

From \eqref{rao} it follows that for $l$ sufficiently large, $\mslt$ splits over
$${\rm K}_l=\{g\in  {\rm SL}_2(\Of) \mid g=I_{2} \, ( \operatorname{mod} \, \Pf^l) \}$$
 via the section $sec$. The topology on $\mslt$ is defined so that $\{sec({\rm K}_l)\}_{l \gg 0}$ form a basis of the neighborhoods of the identity.

 \begin{lem} \label{center and max} Set
$$C=\{h(a) \mid a \in {{\rm F}^*}^d \}$$
and
$$M=\{h(a) \mid a\in J \}.$$
Then $\widetilde{C}$ is the center of  $\widetilde{\rm{H}}$ and $\widetilde{\rm M}$ is a maximal abelian subgroup of $\widetilde{\rm{H}}$.
\end{lem}
\begin{proof}
From \eqref{rao} it follows that $c \bigl(h(a),h(b) \bigr)=(b,a)_n.$ Hence, inverse images of $h(a)$ and $h(b)$ in $\mslt$ commute if and only if $(b,a)_n^2=1$.
By \eqref{FV fact}, this is equivalent to $(a,b)_d=1$.  Both assertions now follow.
 \end{proof}
\subsection{Representations} \label{car rep}
From the cocycle formula \eqref{rao} it follows that $\mslt$ splits over ${\rm N}$ canonically via the section $sec$ and  that $\widetilde{\rm{H}}$ normalizes $sec\bigl( {\rm N} \bigr)$. We shall view ${\rm N}$ as a subgroup of $\mslt$ by identifying it with $sec({\rm N})$. Any representation of $\widetilde{\rm{H}}$
can be extended to a representation of $\widetilde{{\rm B}}$ by defining it to be trivial on  ${\rm N}$. Thus, as in the linear case, we shall not distinguish between representations of $\widetilde{\rm{H}}$  and those of $\widetilde{{\rm B}}$. A (complex) representation of $\mslt$ or any of its subgroups is called genuine if the central subgroup $\mu_n$ acts by the previously fixed injective character $\mu_n \hookrightarrow \C^*.$

Let $\beta$ be a character of  $\widetilde{\rm M}$. We define ${\rm Ind}_ {\widetilde{\rm M}}^{\widetilde{\rm{H}}} \beta$ to be the space of complex functions on $\widetilde{\rm{H}}$ such that
$$f(ah)=\beta(a)f(h)$$ for all $a \in\widetilde{\rm M}, \, h \in \widetilde{\rm{H}}.$ We shall denote by $i(\beta)$ the  representation of $\widetilde{\rm{H}}$ acting on ${\rm Ind}_ {\widetilde{\rm M}}^{\widetilde{\rm{H}}} \beta$ by right translations.  By  the Stone-Von Neumann Theorem, see  Theorem 3.1 in \cite{Weissman09} for example, we have
\begin{lem} \label{svn} The isomorphism class of a genuine smooth irreducible  representation $\sigma$ of $\widetilde{\rm{H}}$  is determined by its central character $\chi_\sigma$. Moreover, a realization of $\sigma$ is given by $i(\chi'_\sigma)$ where $\chi'_\sigma$ is a character of  $\widetilde{\rm M}$ which extends $\chi_\sigma$. In particular, the dimension of $\sigma$ is $[\widetilde{\rm{H}}:\widetilde{\rm M}]=d \ab d \ab^{-\half}.$
\end{lem}
\begin{remark}
In \cite{Weissman14} Weissman refers to
$[\widetilde{\rm{H}}:\widetilde{\rm M}]=\sqrt{[\widetilde{\rm{H}}:\widetilde{C}]}$ as the central index of $\widetilde{\rm{H}}$. It arises in the context of Lagrangian decompositions of $\widetilde{T}/ Z\bigl( \widetilde{T}\bigr)$, where $\widetilde{T}$ is a cover of a torus $T$ defined over a local field and $Z \bigl( \widetilde{T}\bigr)$ is its center. In the case of $\mslt$ these Lagrangian decompositions of $\widetilde{\rm{H}}/\widetilde{C}$ are in a natural bijection with the Lagrangian decompositions of ${\rm F}^*/{{\rm F}^*}^d$ discussed in Section \ref{npower}.
\end{remark}
We shall use  $\sigma$   to denote a genuine smooth irreducible representation of  $\widetilde{\rm{H}}$ with a central character $\chi_{\sigma}$ acting on a representation space $V$ and we shall denote by $V^*$ the space of functionals on $V$. For $s \in \C$ we define $\sigma_s$ to be the smooth irreducible  representation of  $\widetilde{\rm{H}}$ acting on $V$ by
$$\sigma_s(t)=\ab t \ab^s \sigma (t),$$
where for $t=\bigl(h(a),\epsilon \bigr) \in \widetilde{\rm{H}}$ we set $\ab t \ab=\ab a \ab$. We also define $\sigma^w$ to be the representation of  $\widetilde{\rm{H}}$ acting on $V$ by
$$t \mapsto \sigma^w(t)=\sigma(wtw^{-1}).$$
We note here that by the cocycle formula, \eqref{rao},  $$w\bigl(h(a), \epsilon\bigr )w^{-1}= \bigl(h(a^{-1}), \epsilon\bigr ).$$
In particular, ${(\sigma_s)}^w$ is also a genuine smooth irreducible representation of  $\widetilde{\rm{H}}$ and
\begin{equation} \label{sigmainter} {(\sigma_s)}^w={(\sigma^w)}_{-s}.\end{equation}
We now consider the genuine principal series representation induced from $\sigma$, defined as usual by
$${\rm I}(\sigma)={\rm Ind}^{\mslt}_{\mb} \delta^{\half} \otimes \sigma$$
Where $\delta (t)=\ab t \ab^2$ is the modular function.

Given $f \in {\rm I}(\sigma)$ and $s \in \C$  we define $f_s$ to be the following function on $\widetilde{\rm{G}}$. Given $g \in \widetilde{\rm{G}}$ we pick $t \in \widetilde{\rm{H}}, \, n \in {\rm N}$ and $k$ an inverse image inside $\widetilde{\rm{G}}$ of an element of ${\rm SL}_2(\Of)$ such that $g=tnk$ and we set
\begin{equation} \label{flatsec} f_s(g)=\ab t \ab^sf(g). \end{equation} As in the linear case one verifies, using the Iwasawa decomposition of $G$, that $f_s$ is a well defined element in ${\rm I}(\sigma_s)$ and that $f\mapsto  f_s$ is an isomorphism of vector spaces. The following is also proven as in the linear case. See for example Section 7 of \cite {Mc}  or Chapter 5 of \cite{BanPhD}.
\begin{lem} \label{inter} Fix $f\in {\rm I}(\sigma)$ and $g \in \widetilde{\rm{G}}$. the integral
\begin{equation} \nonumber \int_F f_s\bigl(w n(x)g \bigr) \, d_{\psi}x \end{equation} converges absolutely to a $V$ valued rational function in $q^{-s}$ provided that $\ab \chi_\sigma\bigl(h(\varpi^d),1 \bigr)\ab <q^{Re(s)d}$ . We shall denote its meromorphic continuation by  $\bigl({\rm A}_{w}(\sigma_s)(f_s)\bigr)(g)$. Away from its poles,

$${\rm A}_{w}(\sigma_s) \in \operatorname{Hom}_{\mslt} \bigl({\rm I}(\sigma_s) ,{\rm I}((\sigma_s)^w) \bigr).$$
\end{lem}
\subsection{Whittaker functionals} \label{wf}
Let $\psi$ be a character of ${\rm F}$. We view $\psi$ also as a character of ${\rm N}$ by setting $\psi \bigl(n(x)\bigr)=\psi(x).$ Given a representation $(\pi,W)$ of $\mslt$ we shall denote by $\operatorname{Wh}_\psi \bigl(\pi)$ the space of $\psi$-Whittaker functionals on $(\pi,W)$. Namely,
$$\operatorname{Wh}_\psi (\pi)=\operatorname{Hom}_{\rm N}(W,\C_\psi).$$

\begin{lem} \label{wellknown} Let $\sigma$ be a genuine smooth irreducible representation of $\widetilde{\rm{H}}$ acting on $V$. Denote its central character by $\chi_\sigma$ and denote the space of linear functionals on $V$ by $V^*$.\\
1. Assume that  $\ab \chi_\sigma\bigl(h(\varpi^d),1 \bigr)\ab <q^{Re(s)d}$. Fix  $\xi \in V^*$. For $f_s \in {\rm I}(\sigma_s)$ the integral
$$\int_F \xi \Bigl(f_s \bigl(wn(x) \bigr)\Bigr) \, \psi^{-1}(x) \, d_{\psi}x$$ converges absolutely to a polynomial in $q^{-s}$. Moreover, this integral converges in principal value for all $s$. Namely,  $$\lim_{r \rightarrow \infty} \int_{\Pf^{-r}} \xi \Bigl(f_s \bigl(wn(x) \bigr)\Bigr) \, \psi^{-1}(x) \, d_{\psi}x$$ exists for all $s$. \\
2. Denote by $\bigl(J_{\sigma_s,\psi}(\xi)\bigr)(f_s)$ the analytic continuation of the integral defined above. Then, $J_{\sigma_s,\psi}(\xi) \in {\rm Wh}_{\psi}\bigl({\rm I}(\sigma_s) \bigr)$.\\
3. The map $\xi \mapsto J_{\sigma_s,\psi}(\xi)$ is an isomorphism from $V^*$ to  ${\rm Wh}_{\psi}\bigl({\rm I}(\sigma_s) \bigr)$. In particular,
\begin{equation} \label{whidim} \dim \, {\rm Wh}_{\psi}\bigl({\rm I}(\sigma_s) \bigr)=d \ab d \ab^{-\half}.\end{equation}
\end{lem}
\begin{proof} This Lemma is  proven in Theorem 7 in \cite{Mc2} where the author realizes the inducing representation as induction from a particular maximal abelian subgroup of the metaplectic torus.
\end{proof}
\begin{remark} \label{jac id} Note that since $\sigma^w$ acts on $V$ as well, Lemma \ref{wellknown} gives a natural identification between ${\rm Wh}_{\psi}\bigl({\rm I}(\sigma_s) \bigr)$ and ${\rm Wh}_{\psi}\bigl({\rm I}((\sigma_s)^w) \bigr)$ as both spaces are naturally isomorphic to $V^*$. \end{remark}
\section{Shahidi local coefficient matrix} \label{meta sha}
\subsection{Definition} \label{fangao}
Given two representations $\pi$ and $\varsigma$ of $\mslt$ and   $T \in \operatorname{Hom}_{\mslt} \bigl(\pi, \varsigma \bigr)$ one defines by duality
$$T^{\psi}:\operatorname{Wh}_\psi (\varsigma)\rightarrow \operatorname{Wh}_\psi (\pi).$$
Precisely, we set $T^\psi(\xi)=\xi \circ T.$
\begin{remark} \label{plan is lc} If $\pi=\varsigma$ and $A$ is a scalar map corresponding to $\mu \in \C$ then $A^{\psi}$ is also a scalar map corresponding to $\mu$.
\end{remark}
\begin{deff} \label{thedeff} Let $(\sigma,V)$ be a smooth irreducible representation of  $\widetilde{{\rm H} }$.  An Slcm associated with $\sigma$ and $\psi$ is an $\sqrt{[{\rm F}^*:{{\rm F}^*}^d]} \times \sqrt{[{\rm F}^*:{{\rm F}^*}^d]}$ matrix representing $$A^\psi_{w}(\sigma_s):{\rm Wh}_{\psi}\bigl({\rm I}((\sigma_s)^w)  \bigr) \rightarrow {\rm Wh}_{\psi}\bigl({\rm I}(\sigma_s) \bigr)$$
with respect to the ordered bases $J_{(\sigma_s)^w,\psi}(R)$ and $J_{\sigma_s,\psi}(R)$ where $R$ is an ordered basis of $V^*$.
\end{deff}
Note that  an Slcm is associated not only with the isomorphism class of $\sigma$ and a character $\psi$ of ${\rm F}$ but also with a particular realization of the inducing representation and with a choice of an ordered basis $R$ above. However, our construction implies the following.
\begin{prop} \label{withgao} An Slcm associated with $\sigma$ and $\psi$ is unique up to conjugations  by complex matrices. Consequently, the conjugacy class of an Slcm associated with $\sigma$ and $\psi$ is a well defined invariant of $\psi$ and the isomorphism class of  $\sigma$.
\end{prop}
\begin{deff} \label{trace and det}  Define $T(\sigma,s,\psi)$ and $D(\sigma,s,\psi)$ respectively to be the trace and determinant  of an Slcm associated with $\sigma$ and $\psi$. By Proportion \ref{thedeff} these are well defined invariants of $\sigma$ and $\psi$. In particular, $T(\sigma,s,\psi)$ and $D(\sigma,s,\psi)$ depend only on $\psi$ and $\chi_\sigma$.
\end{deff}

\begin{remark} Since the unipotent radical of ${\rm H}$ is trivial, $V^*$ is the space of Whittaker functionals on $\sigma$. Thus, the integral of Jacquet-type given in Lemma \ref{wellknown} identifies the space of Whittaker functionals on $\sigma$ with the space of Whittaker functionals on ${\rm I}(\sigma_s)$. The same construction applies for parabolic induction on other metaplectic groups where the parabolic subgroup in discussion is not necessarily minimal. A similar identification is already implicit in the work of Shahidi on quasi split reductive groups, \cite{Sha1}, where the space of Whittaker functionals on the inducing representation is at most one dimensional. Therefore, one may generalize  Definitions \ref{thedeff} and \ref {trace and det} for general parabolic induction on metaplectic groups such that an analog of Proposition \ref{withgao} holds.
\end{remark}
\subsection{A convenient model} \label{modelsection}
From this point we assume that $n$ is not divisible by 4. This assumption implies that $\# \overline{J}$, $\# \overline{K}$ and $[{\rm F}^*:{{\rm F}^*}^d]$ are odd. For a character $\chi$ of ${\rm F}^*$ we define $$\chi_\psi :\widetilde{\rm{H}} \rightarrow \C$$ by
$$\chi_\psi \bigl(h(a),\epsilon)=\epsilon \chi(a) \begin{cases} 1 & n \, \operatorname{is} \,  \operatorname{odd };\\  \gamma_\psi(a)^{-1}  & n  \equiv 2 \, (\operatorname{mod }4). \end{cases}$$
We note that $\chi_\psi$ is not a character of $\widetilde{\rm{H}}$. Rather it is a quasi linear character in the  sense defined and studied in \cite{AG}.  When convenient we shall also denote by $\chi_\psi $,   the function on ${\rm F}^*$  defined by $a \mapsto  \chi_\psi\bigl((h(a),1\bigr).$

Denote the restriction of $\chi_\psi$ to $\widetilde{\rm M}$ and to $\widetilde{C}$ by $\chi'_\psi $ and $\chi''_\psi $  respectively. Observe that by Lemma \ref{res hil} and by \eqref{weil prop}, $\chi'_\psi $ and $\chi''_\psi $ are genuine characters of $\widetilde{\rm M}$ and $\widetilde{C}$  respectively.
Moreover, the set of genuine characters of  $\widetilde{C}$ is a principal homogeneous space, i.e. a torsor, for the group of characters of ${{\rm F}^*}^d$ so all the genuine characters of  $\widetilde{C}$ are obtained by twisting one genuine character by a character of ${{\rm F}^*}^d$. Since all the characters of  ${{\rm F}^*}^d$ arise from restricting characters of ${\rm F}^*$ it follows that all the genuine characters of $\widetilde{C}$ are of the form  $\chi''_\psi $ for some character $\chi$ of ${\rm F}^*$. Similar reasoning applies to   $\widetilde{\rm M}$.

Given $\sigma$, a genuine smooth irreducible  representation of  $\widetilde{\rm{H}}$ whose central character is $\chi_\sigma$ we choose a character $\chi$ of ${\rm F}^*$ such that $\chi''_\psi=\chi_\sigma$. Utilizing Lemma  \ref{svn} we assume without loss of generality that $$(\sigma,V)=\bigl(i(\chi'_\psi),{\rm Ind}_ {\widetilde{\rm M}}^{\widetilde{\rm{H}}} \chi'_\psi\bigr).$$ According to our notations, $\sigma_s$ and $(\sigma_s)^w$ act on  the same space as $\sigma$ by
$$\bigl(\sigma_s(t)f\bigr)=\ab t \ab^s f(gt), \quad \bigl((\sigma_s)^w(t)f\bigr)=\ab t \ab^{-s} f(gt^w).$$
For $k \in K$ we define $\xi_{\chi,k} \in \bigl({\rm Ind}_ {\widetilde{\rm M}}^{\widetilde{\rm{H}}} \chi'_\psi\bigr)^*$ by $$\xi_{\chi,k}(f)=\chi_\psi(k)^{-1}f\bigl(h(k),1).$$
Observe that for $c \in {{\rm F}^*}^d$ we have $\xi_{\chi, kc}=\xi_{\chi,k}$. Thus, since Remark \ref{nicedid} implies that $\overline{K} \cong \widetilde{\rm{H}}/\widetilde{\rm M}$
it follows that
$$R_\chi=\{\xi_{\chi,k} \mid k \in \overline{K}\}$$ is a well defined  ordered basis for $\bigl({\rm Ind}_ {\widetilde{\rm M}}^{\widetilde{\rm{H}}} \chi'_\psi\bigr)^*$.  In Section \ref{thecomp} we shall compute the matrix representing
$$A^\psi_{w}(\sigma_s):{\rm Wh}_{\psi}\bigl({\rm I}((\sigma_s)^w)  \bigr) \rightarrow {\rm Wh}_{\psi}\bigl({\rm I}(\sigma_s) \bigr)$$
with respect to $J_{(\sigma_s)^w,\psi}(R_{\chi})$ and $J_{\sigma_s,\psi}(R_{\chi})$. We note that this Slcm is determined by the Lagrangian decomposition $L=(J,K)$ and by $\chi$ and that it is a function on $\overline{K} \times \overline{K}$. We shall denote it by
$$\tau_{_L}(\cdot,\cdot,\chi,s,\psi):\overline{K} \times \overline{K} \rightarrow \C[q^{-s}].$$

The rest of this section deals with the technical issue of finding a convenient model for the principal series representation in discussion. We shall write $\chi_{\psi,s}$, $\chi'_{\psi,s}$ and $\chi''_{\psi,s}$ for ${(\chi_s)}_\psi$, ${(\chi'_s)}_\psi$ and ${(\chi''_s)}_\psi$ respectively. We shall think of the representation space of ${\rm I}\bigl(\sigma_s\bigr)$  as the space of functions $$h: \widetilde{\rm{H}} \times \mslt \rightarrow \C$$ which are smooth from the right in the right argument and which satisfy
$$h(t^+t_0,tng)=\ab t \ab^{s+1}\chi'_\psi(t^+)h(t_0t,g)$$
for all $t,t_0 \in \widetilde{\rm{H}}, \, t^+ \in \widetilde{\rm M}, \, n\in {\rm N}, \, g\in \mslt.$ Similarly we think of the representation space of ${\rm I}\bigl((\sigma_s)^w\bigr)$ as the space of functions $$r:\widetilde{\rm{H}} \times \mslt  \rightarrow \C$$ which are smooth from the right in the right argument and which satisfy
$$r(t^+t_0,tng)=\ab t \ab^{-s+1}\chi'_\psi(t^+)r(t_0t^w,g)$$
for all $t,t_0 \in \widetilde{\rm{H}}, \, t^+ \in \widetilde{\rm M}, \, n\in {\rm N}, \, g\in \mslt.$
$\mslt$ acts on both spaces by right translations on the right argument.

Conforming with the standard formalism in the literature, dating back to \cite{KP}, we now introduce another model for the principal series representations, namely, induction by stages. Define
$$\operatorname{Ind}_{\widetilde{\rm M}{\rm N}}^{\mslt} \delta^\half \otimes \chi'_{\psi,s}.$$
to be space of smooth from the right functions   $$g_s:\mslt \rightarrow \C$$ such that
$$g_s(tng)=\chi'_{\psi,s}(t)\ab t \ab f(g)$$
for all $t \in  \widetilde{\rm M}, \, n\in {\rm N}, \, g\in \mslt.$ Denote by  ${\rm I}\bigl(\chi'_{\psi,s}\bigr)$  the representation of $\mslt$ acting on $\operatorname{Ind}_{\widetilde{\rm M}{\rm N}}^{\mslt} \delta^\half \otimes \chi'_{\psi,s}$ by right translations. We have the natural $\mslt$ isomorphisms
$$M_s:{\rm I}\bigl(\sigma_s\bigr) \rightarrow{\rm I}\bigl(\chi'_{\psi,s}\bigr), \quad N_s:{\rm I}\bigl({\chi'^{-1}}_{\psi,-s}\bigr) \rightarrow{\rm I}\bigl((\sigma_s)^w\bigr)$$
defined by
$$\bigl(M_s(h)\bigr)(g) = h\bigl((I_2,1),g\bigr), \quad \bigl(N_s(f)\bigr)(t,g) =\ab t \ab^{1-s} f(t^wg).$$
Their inverses
$$(M_s)^{-1}:  {\rm I}\bigl(\chi'_{\psi,s}\bigr)   \rightarrow{\rm I}\bigl(\sigma_s\bigr), \quad (N_s)^{-1}:{\rm I}\bigl((\sigma_s)^w\bigr) \rightarrow{\rm I}\bigl({\chi'^{-1}}_{\psi,-s}\bigr) $$
are given by
$$\bigl((M_s)^{-1}(f)\bigr)(t,g)=\ab t \ab^{-1-s} f(tg), \quad \bigl((N_s)^{-1}(r)\bigr)(g) = r\bigl((I_2,1),g\bigr).$$

\begin{lem} \label{comdia}The following  diagram is commutative.
$$\begin{CD}
{\rm I}\bigl(\chi'_{\psi,s}\bigr)  @<{M_s}<<{\rm I}\bigl(\sigma_s\bigr)  \\
@V{  {\rm A}_w\bigl(\chi'_{\psi,s}\bigr)}VV    @V{ {\rm A}_w(\sigma_s)}VV \\
{\rm I}\bigl({\chi'^{-1}}_{\psi,-s}\bigr) @>{N_s}>> {\rm I}\bigl((\sigma_s)^w\bigr)
\end{CD}$$
Here, for $h_s=M_s(f_s) \in {\rm I}\bigl(\chi'_{\psi,s}\bigr)$ and $g \in \widetilde{\rm{G}}$,  $ \bigl({\rm A}_w\bigl(\chi'_{\psi,s}\bigr)(h_s)\bigl)(g)$ is the meromorphic continuation of
$$ \int_F h_s\bigl(wn(x)g \big) \, d_\psi x.$$
This integral converges absolutely wherever $A(\sigma_s)$ converges absolutely.
\end{lem}
\begin{proof}
Given Lemma \ref{inter} one only needs to show that wherever  ${\rm A}_w(\sigma_s)$ is given by an absolutely convergent integral,
$$\bigl( (N_s)^{-1} \circ {\rm A}_w(\sigma_s) \circ (M_s)^{-1} (f_s)\bigr)(g)=\int_F f_s\bigl(wn(x)g \big) \, d_\psi x.$$
This is a matter of a direct computation.
\end{proof}

The diagram in Lemma \ref{comdia} gives rise to the dual commutative diagram
$$\begin{CD}
\operatorname{Wh}_\psi \bigl({\rm I}\bigl(\chi'_{\psi,s}\bigr)\bigr)  @>{M^\psi_s}>> \operatorname{Wh}_\psi \bigl( {\rm I}\bigl(\sigma_s\bigr)\bigr)  \\
@A{  A^\psi_w\bigl(\chi'_\psi,s\bigr)}AA    @A{ {\rm A}_w^\psi(\sigma_s)}AA \\
\operatorname{Wh}_\psi \bigl({\rm I}\bigl({\chi'^{-1}}_{\psi,-s} \bigr)\bigr) @<{N^\psi_s}<< \operatorname{Wh}_\psi \bigl( {\rm I}\bigl((\sigma_s)^w\bigr)\bigr)
\end{CD}$$
Fix $k \in \overline{K}$. We define  $\lambda_{k,\chi,\psi,s} \in \operatorname{Wh}_\psi \bigl({\rm I}\bigl(\chi'_{\psi,s}\bigr)\bigr)$ by
\begin{equation} \label{lamup} \lambda_{k,\chi,\psi,s}=\bigl((M_s^\psi)^{-1} \circ J_{\sigma_s,\psi}  \bigr) (\xi_{\chi,k}). \end{equation}

\begin{lem} \label{forlambda} Fix $k \in \overline{K}$. Let $y$ be any representative of $k$ in $K$. $\lambda_{k,\chi,\psi,s}$ is the analytic continuation of
 $$h_s=M_s(f_s) \mapsto \chi_{\psi,s+1}(y)^{-1} \int_{\rm F} h_s\bigl((h(y),1)wn(x)\bigr)\psi^{-1}(x) \, d_{\psi}x.$$
This integral converges absolutely if  $\ab \chi\bigl(\varpi^d) \ab <q^{Re(s)d}$ to a polynomial in $q^{-s}$ and converges in principal value for all $s$. Namely, for all $s \in \C$ and $h_s=M_s(f_s) \in{\rm I}\bigl(\chi'_{\psi,s}\bigr)$
$$\lim_{r \rightarrow \infty} \int_{\Pf^{-r}} h_s\bigl((h(y),1)wn(x)\bigr)\psi^{-1}(x) \, d_{\psi}x $$ exists.
\end{lem}
\begin{proof} A straightforward computation shows that wherever $J_{\sigma_s,\psi}$ is given by an absolutely convergent integral we have
$$\bigl((M_s^\psi)^{-1} \circ J_{\sigma_s,\psi}  (\xi_{\chi,k}) \bigr)(h_s) = \chi_{\psi}(y)^{-1} \ab y \ab^{-s-1}\int_{F} h_s\bigl((h(y),1)wn(x)\bigr)\psi^{-1}(x) \, d_{\psi}x.$$
(the factor $ \ab y \ab^{-s-1}$ is coming from the definition of $(M_s)^{-1}$ while the factor  $\chi_{\psi}(y)^{-1}$ originates from the normalization of $\xi_{\chi,k}$). This lemma now follows from Lemma \ref{wellknown}.
\end{proof}
\begin{lem} \label{lamdown}
$$N_s^\psi \circ J^\psi_{{\sigma_s}^w} (\xi_{\chi,k})=\lambda_{k^{-1},\chi^{-1},\psi,-s}.$$
\end{lem}
\begin{proof}  Fix a representative $y$ of $k$ in $K$. By a straightforward computation, almost identical to the one used in Lemma \ref{forlambda} one shows that wherever $J_{(\sigma_s)^w,\psi}$ is given by an absolutely convergent integral, for $h_{-s}=M_{-s}(f_{-s})$ we have
$$\bigl(N_s^\psi\circ J_{(\sigma_s)^w,\psi}  (\xi_{\chi,k}) \bigr)(h_{-s}) = \chi_{\psi}(y)^{-1} \ab y \ab^{-s+1}\int_{F} h_{-s}\bigl((h(y^{-1}),1)wn(x)\bigr)\psi^{-1}(x) \, d_{\psi}x$$
(the inversion of $y$ inside the integral is due to the conjugation by $w$ in the definition of $N_{s}$). Recalling the definition of $\chi_\psi$ and taking into account the fact that $\gamma_\psi(y)=\gamma_\psi(y^{-1})$ we obtain
$$\chi_{\psi}(y)^{-1} \ab y \ab^{-s+1}=\bigl((\chi^{-1})_{\psi,1-s}(y^{-1})\bigr)^{-1}.$$
Thus, using Lemma \ref{forlambda} this lemma follows wherever the integral defining $ J_{(\sigma_s)^w,\psi}$ is absolutely convergent. By means on analytic continuation, it is now proven for all $s$.
\end{proof}
Combining \eqref{lamup} and Lemma \ref{lamdown} we have proven the following
\begin{prop}
\begin{equation} \label{mykpmatrix}A^\psi_w\bigl(\chi'_\psi,s\bigr)\bigl(\lambda_{a^{-1},\chi^{-1},\psi,-s}\bigr)=\sum_{b \in \overline{K}} \tau_{_L}(a,b,\chi_{},s,\psi)\lambda_{b,\chi,\psi,s}. \end{equation}
\end{prop}
\begin{remark} \label{ichange} Equation \eqref{mykpmatrix} is not an exact analog of the matrix defined in the bottom of Page 75 of \cite{KP} for unramified representations of coverings of ${\rm GL}_n$ using the Lagrangian decomposition given in Example \ref{kp example} and the conjugacy class of the matrix it produces is not equal to the conjugacy class of the matrix computed in \cite{Mc2}. The modifications we made are the following: first, we have $a^{-1}$ rather than $a$ in the left hand side of \eqref{mykpmatrix}. Second, we used the normalization factor $\ab k \ab^s$ in the definition of $\lambda_{k,\chi,\psi,s}$. Both modifications are crucial for the invariance property of the Slcms given in Proposition \ref{withgao} and are imposed on us once we work with delta functionals on $\bigl({\rm Ind}_ {\widetilde{\rm M}}^{\widetilde{\rm{H}}} \chi'_\psi\bigr)^*$. The first modification reflects the action of the Weyl group on the metaplectic torus. We note here that a similar swap of the rows of the scattering matrix utilized in Theorem 1 of \cite{BBB} takes place. Moreover, that swap is also due to the action of the Weyl group.  Last, the normalization factor $\chi_\psi$ in the definition of  $\lambda_{k,\chi,\psi,s}$ is responsible for the fact that the Slcm we study is a function on $\overline{K} \times \overline{K}$ rather than a function on a fixed set of representatives of $\widetilde{\rm{H}}/\widetilde{\rm M} \times \widetilde{\rm{H}}/\widetilde{\rm M}$. This last property facilitates some of the proofs below and simplifies the formula given in Theorem \ref{metasha} below. This simplification is the main reason why we work with a Lagrangian decomposition and not only with a Lagrangian subgroup.
\end{remark}
\begin{remark} \label{for T prop} Recall that if $n$ is odd then $\chi_\psi$ is independent of $\psi$. If $n\equiv 2 \, (\operatorname{mod }4)$
we use \eqref{weilchangepsi} and deduce that
$$\chi_{\psi_x}=\bigl((\cdot,x)_2\chi\bigr)_\psi.$$ This implies that if $n$ is odd and $\tau_{_L}(a,b,\chi_{},s,\psi)$ is an Slcm associated with $\sigma$ and $\psi$ then $\tau_{_L}(a,b,\chi_{},s,\psi_x)$ is an Slcm associated with $\sigma$ and $\psi_x$ while if $n\equiv 2 \, (\operatorname{mod }4)$ and $\tau_{_L}(a,b,\chi_{},s,\psi)$ is an Slcm associated with $\sigma$ and $\psi$ then $\tau_{_L}(a,b,\chi_{},s,\psi_x)$ is an Slcm associated with $(\cdot,x)_2 \otimes \sigma$ and $\psi_x$. Here $(\cdot,x)_2 \otimes \sigma$ is the non-genuine quadratic twist of $\sigma$ given by $$\bigl((\cdot,x)_2 \otimes \sigma\bigr)\bigl(h(a),\epsilon\bigr)=(a,x)_2\sigma\bigl(h(a),\epsilon\bigr).$$
\end{remark}
\subsection{Relation with partial $\gamma$ and $\widetilde{\gamma}$-factors} \label{thecomp}
In this section we shall establish a relation between $\tau_{_L}(a,b,\chi_{},s,\psi)$ and the partial $\gamma$ and $\widetilde{\gamma}$-factors. For the  $n\equiv 2 \, (\operatorname{mod }4)$ cases we need the following. Let $\theta$ be an automorphism of $\widehat{{\rm F}^*/ {{\rm F}^*}^d}$. The notion of a Lagrangian subgroup and the results in Sections \ref{npower}-\ref{MT and PMT} are unchanged if we replace $\eta_x$ by $\theta(\eta_x)$. In particular, if $n \equiv 2 \, (\operatorname{mod }4)$ then $d$ is odd and $m=\frac{d+1}{2} \in \N$ is relatively prime to $d$. Thus, if we define
\begin{equation} \label{new eta}  x\mapsto \eta'_x=\eta^m_x \end{equation}
then $\eta_x \mapsto  \eta'_x$ is an automorphism  of the dual group of ${{\rm F}^*/ {{\rm F}^*}^d}$.
\begin{thm} \label{metasha} Let $\sigma$ be a genuine smooth irreducible  representation of  $\widetilde{\rm{H}}$ with a central character  $\chi_\sigma$. Let $\chi$ be a character of ${\rm F}^*$ such that $\chi''_\psi=\chi_\sigma$. An Slcm associated with $\sigma$ and $\psi$ is given by
$$(a,b) \mapsto\tau_{_L}(a,b,\chi_{},s,\psi)=\begin{cases} \gamma_{_J}(1-s,\chi^{-1}\eta_{ab},\psi,ab^{-1}) &  n \, \operatorname{is} \,  \operatorname{odd }; \\ \\ \widetilde{\gamma}_{_J}(1-s,\chi^{-1}\eta'_{ab},\psi,ab^{-1} \bigr)&  n\equiv 2 \, (\operatorname{mod }4). \end{cases}$$
Here $a,b \in \overline{K}$.
\end{thm}
\begin{proof} Fix $b \in \overline{K}$ and $l \gg 0$. Let $f^{0}\in {\rm I}\bigl(\chi'_{\psi,s}\bigr)$ be the function supported in
$$\widetilde{\rm M}{\rm N}sec\bigl( (h(b),1){\rm K}_lw_0\bigr)=\widetilde{\rm M}{\rm N}(h(b),1)w{\rm K}_l^+,$$
where  ${\rm K}_l^+=\widetilde{{\rm K}_l}\cap {\rm N}$, normalized such that
$$f^{0}(an\bigl(h(b),1\bigr)w k)=\operatorname{Vol}_{\psi}(\Pf^l)^{-1}\chi_{\psi,s+1}(a)$$ for all
$a \in \widetilde{\rm M},  n \in {\rm N}$ and $k \in {\rm K}_l^+$. Arguing as in Lemma 1.31 of \cite{KP} one shows that
$$\lambda_{c,\chi,\psi,s}( f^{0} )=\begin{cases} \bigl(\chi_{\psi,s+1}(b) \bigr)^{-1} & b=c ;\\ 0 & b \neq c . \end{cases}$$
Recalling \eqref{mykpmatrix} we obtain
$$\tau_{_L}(a,b,\chi,s,\psi)=\chi'_{\psi, 1+s}(b) \Bigl(A^\psi_w\bigl(\chi'_{\psi,s}\bigr)\bigl(\lambda_{a^{-1},\chi^{-1},\psi,-s}\bigr) \Bigr) ( f^{0} ).$$
We now argue that  for $Re(s) \gg 0$
\begin{eqnarray}  \nonumber &&  \Bigl(A^\psi_w\bigl(\chi'_{\psi,s}\bigr)\bigl(\lambda_{a^{-1},\chi^{-1},\psi,-s}\bigr) \Bigr) (f^{0}) =  \bigl(\chi_{\psi,s-1}(a) \bigr)^{-1} (-a,b)_n \ab ab \ab ^{-1} \times \\ \nonumber && \lim_{r \rightarrow \infty} \int_{Jba^{-1}\cap \Pf^{-r}} (\chi_s)_\psi(zab^{-1}) (z,ab)_n \psi(z) \, d^*_\psi z.
\end{eqnarray}
Indeed, this was proven as Equation (4.2) in Lemma 4.1 of \cite{GoSz} for the case where $n$ is relatively prime to the residual characteristic of ${\rm F}$ and $L$ is as in Example \ref{kp example}. The reader can verify that the argument there applies to the general case as well.

If $n$ is odd then $\chi_\psi=\chi$, $d=n$ and  $(-a,b)_n=1$ as both $a,b \in \overline{K}$. Thus,
$$\tau_{_L}(a,b,\chi_{\psi},s,\psi)=\lim_{r \rightarrow \infty}\int_{Jba^{-1} \cap \Pf^{-r}} \chi_s(z) \eta^{-1}_{ab}(z)  \psi(z) \, d^*_\psi(z).$$
Using the second assertion in Lemma \ref{partial gamma lemma} we conclude the proof of this  theorem for the  case where $n$ is odd.

Suppose now that $n \equiv 2 \, (\operatorname{mod }4)$. In this case, by \eqref{weilchangepsiinv} we have  $$\chi_\psi(x)= \gamma_\psi(x)^{-1} \chi(x)=\chi(x) \gamma_\psi(x)(x,-1)_2.$$ This gives
$$\! \! \! \! \!  \! \! \! \! \! \! \! \! \! \!  \! \! \! \! \! \! \! \! \! \!  \! \! \! \! \!\tau_{_L}(a,b,\chi_{\psi},s,\psi)=(b,-1)_2(a,b)_2(-a,b)_n \times$$ $$\lim_{r \rightarrow \infty} \int_{Jba^{-1}\cap \Pf^{-r}} \gamma_\psi(z)^{-1} \cdot (z,ab)_2 \cdot \chi_s(z) \cdot(z,ab)_n \cdot \psi(z) \, d^*_\psi z.$$
Note now that since $n=2d$ it follows that
$$(z,ab)_2(z,ab)_n=(z,ab)^d_n(z,ab)_n=(z,ab)^{d+1}_n.$$
With the notation introduced in \eqref{new eta} we have
$$(z,ab)^{d+1}_n=\bigl( (z,ab)^{2}_n \bigr)^m=(z,ab)^{m}_d=\eta'^{-1}_{ab}(z).$$
Using similar arguments one shows that $$(b,-1)_2(a,b)_2(-a,b)_n=\eta'_{a}(b).$$
As in the odd case, $\eta'_{a}(b)=1$ for all $a,b \in \overline{K}$. The theorem for the $n \equiv 2 \, (\operatorname{mod }4)$ case now follows from Theorem \ref{partial metagamma lemma}.
\end{proof}
\begin{cor} Let $\sigma$ be a genuine smooth irreducible  representation of  $\widetilde{\rm{H}}$. Then, the entries of an Slcm associated $\sigma$ and $\psi$  lies in the space of rational functions in $q^{-s}$. Moreover, let $\chi$ be a character of ${\rm F}^*$ such that $\chi''_\psi=\chi_\sigma$. If $n$ is odd then the space spanned by the entries  of an Slcm associated with $\sigma$ and $\psi$ is
$$\operatorname{span} \{ s\mapsto \gamma (1-s,\chi^{-1}\eta,\psi) \mid \, \eta \in\widehat{ {\rm F}^*/{{\rm F}^*}^d} \}$$
and if  $n\equiv 2 \, (\operatorname{mod }4)$ then the space spanned by the entries  of an Slcm associated with $\sigma$ and $\psi$ is
$$\operatorname{span} \{ s\mapsto \widetilde{\gamma}  (1-s,\chi^{-1}\eta,\psi) \mid \, \eta \in\widehat{ {\rm F}^*/{{\rm F}^*}^d} \}.$$
\end{cor}
\begin{proof} Due to Proposition \ref{withgao}  it is sufficient to prove  that
\begin{eqnarray} \nonumber && \operatorname{span} \{ s \mapsto \tau(a,b,\chi,s,\psi) \mid \, a,b \in \overline{K} \} \\ \nonumber
&&= \begin{cases} \operatorname{span} \{ s\mapsto \gamma (1-s,\chi^{-1}\eta_x,\psi) \mid \, x \in {\rm F}^*/{{\rm F}^*}^d \} &  n \, \operatorname{is} \,  \operatorname{odd }; \\ \\ \operatorname{span} \{ s\mapsto \widetilde{\gamma} (1-s,\chi^{-1}\eta_x,\psi) \mid \, x \in {\rm F}^*/{{\rm F}^*}^d \}&  n\equiv 2 \, (\operatorname{mod }4). \end{cases} \end{eqnarray}
Indeed, since the order of $\overline{K}$ is odd, the map $a \mapsto a^2$ is an automorphism of $\overline{K}$. We denote by $a\mapsto \sqrt{a}$ its inverse. From Theorem \ref{metasha} it follows that
$$\tau_{_L}( \sqrt{ab},\sqrt{ab^{-1}},\chi_{},s,\psi)=\begin{cases} \gamma_{_J}(1-s, \chi^{-1}\eta_{a},\psi,b )  &  n \, \operatorname{is} \,  \operatorname{odd }; \\ \\  \widetilde{\gamma}_{_J}(1-s, \chi^{-1}\eta_{a},\psi,b )&  n\equiv 2 \, (\operatorname{mod }4). \end{cases}$$
The assertion now follows from \eqref{part gamma def} and \eqref{part meta  gamma def} along with their inversions.
\end{proof}
\begin{cor} \label{more inv}  Let $\sigma$ be a genuine smooth irreducible  representation of  $\widetilde{\rm{H}}$ with a central character  $\chi_\sigma$. For $Re(s) \gg 0$ we have
$$T(\sigma,s,\psi)=  d \ab d \ab^{-\half} \lim_{r \rightarrow \infty} \int_{\Pf^{-r} \cap {{\rm F}^*}^d}
\chi_{\sigma_s}\bigr(h(x),1 \bigl) \psi(x) \, d_\psi^* x .$$
Let $\chi$ be a character of ${\rm F}^*$ such that $\chi''_\psi=\chi_\sigma$. Then,
\begin{equation} \label{traceformula}T(\sigma,s,\psi)=d^{-1} \ab d \ab^{\half}\begin{cases} \sum_{\eta \in \widehat{{\rm F}^*/{{\rm F}^*}^d}} \gamma(1-s,\chi^{-1}\eta,\psi)  &  n \, \operatorname{is} \,  \operatorname{odd }; \\   \sum_{\eta \in \widehat{{\rm F}^*/{{\rm F}^*}^d}} \widetilde{\gamma}(1-s,\chi^{-1}\eta,\psi)&  n\equiv 2 \, (\operatorname{mod }4). \end{cases} \end{equation}
\end{cor}
\begin{proof} We give a proof only for the odd case. The proof for $n\equiv 2 \, (\operatorname{mod }4)$ is similar.
 \begin{eqnarray} \nonumber T(\sigma,s,\psi) &=& \sum_{a \in \overline{K}} \tau(a,a,\chi,s,\psi)=\sum_{a \in \overline{K}} \gamma_{_J} (1-s,\chi^{-1}\eta_{a^2},\psi,1 ) \\ \nonumber &=& \sum_{k \in \overline{K}} \gamma_{_J} (1-s,\chi^{-1}\eta_{k},\psi,1 )=  d \ab d \ab^{-\half} \sum_{k \in \overline{K}} \sum_{j \in \overline{J}}\gamma (1-s,\chi^{-1}\eta_{gk},\psi ) \\ \nonumber &=& d \ab d \ab^{-\half} \sum_{\eta \in \widehat{{\rm F}^*/{{\rm F}^*}^d}} \gamma(1-s,\chi^{-1}\eta,\psi)
 .\end{eqnarray}
To obtain the integral representation note that for $Re(s) \gg 0$
\begin{eqnarray} \nonumber \sum_{a \in {{\rm F}^*/{{\rm F}^*}^d}} \gamma (1-s,(\chi\eta_a)^{-1},\psi \bigr) &=& \sum_{a \in {{\rm F}^*/{{\rm F}^*}^d}}\lim_{r \rightarrow \infty} \int_{\Pf^{-r}} \chi_{s}(x)\eta_a(x) \psi(x) \, d_\psi^* x \\ \nonumber &=&  \lim_{r \rightarrow \infty} \int_{\Pf^{-r}} \Bigl(\sum_{a \in {{\rm F}^*/{{\rm F}^*}^d}} \eta_a(x) \Bigr)   \chi_{s}(x) \psi(x) \, d_\psi^* x. \end{eqnarray}
We finish by using the fact that for $x \in {\rm F}^*$,
$$\sum_{a \in {{\rm F}^*/{{\rm F}^*}^d}}\eta_a(x)= \begin{cases} [{\rm F}^*:{{\rm F}^*}^d]  &  x \in {{\rm F}^*}^d; \\ 0  &   \operatorname{otherwise}. \end{cases}$$
\end{proof}
\subsection{Computations in the case where ${\rm gcd}(n,p)=1$} \label{uram examples}
In this section only we assume that $n>1$ is relatively prime to the residual characteristic of ${\rm F}^*$, that $\psi$ is normalized and that $L$ is the Lagrangian decomposition of ${\rm F}^*/{{\rm F}^*}^d$ given in Example \ref{kp example}. In particular, $\overline{J}$ and $\overline{K}$  are cyclic groups of order $d$.  To demonstrate the effectiveness of our  construction we give in this section the computations of $\gamma_{_J}(s,\chi,\psi,a )$, $\widetilde{\gamma}_{_J}(s,\chi,\psi,a )$ and $\tau_{_L}(a,b,\chi,s,\psi)$. These computations are not needed for Section \ref{planme}. Recall that in the case where ${\rm gcd}(n,p)=1$, $\mslt$ splits uniquely over ${\rm SL}_2(\Of)$, see \cite{Kub}. Thus, if $\chi^n$ is unramified, $\tau_{_L}(a,b,\chi,s,\psi)$  is associated with an unramified genuine smooth irreducible  representation of  $\widetilde{\rm{H}}$. This unramified computation was already utilized in \cite{GSS}. It should also be of interest for  global applications. The relation between the unramifed Slcms computed here and the matrices computed in \cite{Mc2} for $\mslt$ is discussed in Remark \ref{ichange} above.  In the odd case, similar computations can be used  for the scattering matrix in \cite{KP} for both odd and even fold covers of ${\rm GL}_n({\rm F})$. Namely,  $\widetilde{\gamma}_{_J}(s,\chi,\psi,a \bigr)$ does not appear in the formulas for these covering groups. We note here that under the assumptions above, objects closely related to $\gamma_{_J}(s,\chi,\psi,a )$, $\widetilde{\gamma}_{_J}(s,\chi,\psi,a )$ and $\tau_{_L}(a,b,\chi,s,\psi)$ were computed in $\cite{GoSz}$ for all $n$. We start by collecting some information arising from the fact that ${\rm gcd}(n,p)=1$.
\begin{lem} \label{gcd1 stuff} Let $\chi$ be a ramified character of ${\rm F}^*$. The following hold.
\begin{enumerate}
  \item \label{1+pf in fn} If ${\rm gcd}(p,n)=1$ then $1+\Pf \subset {{\rm F}^*}^{n}$.
  \item \label{cond1} If ${\rm gcd}(p,n)=1$ and the restriction of $\chi$ to  ${{\rm F}^*}^{n}$ is trivial then $e(\chi)=1$.
   \item \label{weiltrival} If $p$ is odd and $\psi$ is normalized then $\gamma_F(\psi)=1$.
  \item \label{nocondchange} If ${\rm gcd}(p,n)=1$ and $\chi^n$ is ramified then $e(\chi)=e(\chi\eta)$ for any character $\eta$ which is trivial on ${{\rm F}^*}^{n}$.
   \item \label{chichi2} If ${\rm gcd}(n,p)=1$ and $\chi^n$ is ramified then $e(\chi)=e(\chi^n)$.  In particular, if $p$ is odd and $\chi^2$ is ramified then $e(\chi)=e(\chi^2)$.
\end{enumerate}
\end{lem}
\begin{proof}
The first assertion  is well known, see for example Page 43 of \cite{Lang}. The second assertion follows at once from the first assertion. The third assertion is also well known, see Corollary 3.2 in \cite{CC} for example.

We prove the fourth assertion. Since by the second assertion $e(\eta)\leq 1$, this assertion is trivial in the case where  $e(\chi)>1$. Suppose now that $e(\chi)=1$. In this case
it is sufficient to show that $(\chi \eta)^n$ is ramified. This follows from the fact that the restrictions of $\chi^n$ and $(\chi \eta)^n$ to $\Of^*$ are equal along with the assumption that $\chi^n$ is ramified.

We prove the fifth assertion. Since $\chi^n$ is ramified, $\chi$ restricts and descends to
$$\widetilde{\chi}:1+\Pf^{e(\chi^n)} / 1+\Pf^{e(\chi)} \rightarrow \mu_n.$$
Since $1+\Pf^{e(\chi^n)}/ 1+\Pf^{e(\chi)}$ is a $p$-group and ${\rm gcd}(p,n)=1$ we conclude that $\widetilde{\chi}$ is trivial. Thus, $e(\chi^n)=e(\chi)$.
\end{proof}
Let $u$ and $k$ be primitive elements in $\overline{J}$ and $\overline{K}$ respectively. We shall assume that $k$ is the image of $\varpi$ in ${\rm F}^*/{{\rm F}^*}^d$. Observe that $\eta_u$ is unramified.
Let $\xi$ be the $d^{th}$ root of 1 defined by
$$\xi= \eta_u(\varpi).$$ 
In the proof of Propositions \ref{un ram gama} and \ref{meta un ram gama}  below we shall use the elementary  identities
\begin{equation} \label{elemntry} \prod_{m=0}^{d-1}(1-x \xi^{-m})=1-x^d, \, \, \, \, \prod_{\substack  {0 \leq m \leq d-1 \\ m\neq l}} \! \! \!  \bigl(1-x\xi^{-m}\bigr)=\sum_{m=0}^{d-1} x^m \xi^{-lm} \end{equation}
These identities follow from the fact that $\xi$ is a primitive element of $\mu_d$. In the proof of Propositions \ref{un ram gama} and \ref{meta un ram gama}  below we shall also use the following notation. For $a,b \in \Z$ set
$$\delta_{a,b}= \begin{cases} 1 & a\equiv b \, (\operatorname{mod }d); \\ 0  &   a \not \equiv b \, (\operatorname{mod }d).  \end{cases}$$
\begin{prop} \label{un ram gama}Suppose that $n$ is odd and that  ${\rm gcd}(p,n)=1$. Let $\psi$ be a normalized character of ${\rm F}$ and let $L$ be the Lagrangian decomposition of ${\rm F}^*/{{\rm F}^*}^d$ given in Example \ref{kp example}.

If $\chi$ is ramified then
\begin{equation} \label{ram gama} \gamma_{_J}(1-s,\chi^{-1},\psi, k^{t})=\begin{cases} \epsilon(1-s,\chi^{-1},\psi)  & t\equiv e(\chi) \, (\operatorname{mod }n) ; \\ \\ 0  & otherwise .\end{cases} \end{equation}

If $\chi$ is unramified then
\begin{equation} \label{ram gama un}\gamma_{_J}(1-s,\chi^{-1},\psi,k^{-t} )= \bigl(q^{-s}\chi(\varpi)\bigr)^t \begin{cases} (1-q^{-1}) L(ns,\chi^n)  & 0 \leq t \leq n-2 ; \\  \gamma(1-ns,\chi^{-n},\psi)& t=n-1.  \end{cases} \end{equation}
\end{prop}
\begin{proof}
By \eqref{part gamma def}
\begin{eqnarray}\label{beforecommon}
\gamma_{_J}(1-s,\chi^{-1},\psi,k^{t} )&=&
\frac 1 n \sum_{j \in \overline{J}} \gamma(1-s,\chi^{-1}\eta_j,\psi)\eta_j^{-1}\bigl(k^{t}\bigr) \\ \nonumber &=& \frac 1 n \sum_{l=0}^{n-1} \ \gamma(1-s,\chi^{-1}\eta_u^l,\psi)\eta_u^{-lt}(k)= \frac 1 n \sum_{l=0}^{n-1} \ \gamma(1-s,\chi^{-1}\eta_u^l,\psi)\xi^{-tl}. \end{eqnarray}
Suppose first that $\chi$ is ramified. In this case $$\gamma(1-s,\chi^{-1}\eta_u^l,\psi)=\epsilon(1-s,\chi^{-1}\eta_u^l,\psi).$$ From \eqref{epsilon twist} it follows that
$$\gamma_{_J}(1-s,\chi^{-1},\psi,k^t )=\epsilon(1-s,\chi^{-1},\psi) \frac 1 n \sum_{l=0}^{n-1} \xi^{l(e(\chi)-t)}=\epsilon(1-s,\chi^{-1},\psi)\delta_{t,e(\chi)}.$$
The ramified case is done.

We move to the case where $\chi$ is unramified. In this case
$$\gamma(1-s,\chi^{-1}\eta_u^l,\psi)= \frac { L(s,\chi\eta_u^{-l})}{L(1-s,\chi^{-1}\eta_u^l)} =\frac{1-q^{s-1}\chi^{-1}(\varpi)\xi^l}  {1-q^{-s}\chi(\varpi)\xi^{-l}}.$$
We take a common denominator in the right hand side of \eqref{beforecommon} and then use  \eqref{elemntry}.
\begin{eqnarray} \nonumber &\gamma_{_J}& \! \! \! \! \! \!(1-s,\chi^{-1},\psi,k^{-t}) \\ \nonumber &=& \frac 1 n \prod_{l=0}^{n-1} \! \! (1-q^{-s}\chi(\varpi) \xi^{-l})^{-1} \Bigl( \sum_{l=0}^{n-1} (1-q^{s-1}\chi^{-1}(\varpi)\xi^{l})\xi^{tl}  \! \! \! \prod_{\substack  {0 \leq m \leq n-1 \\ m\neq l}}(1-q^{-s}\chi(\varpi)\xi^{-m}) \Bigr)
\\ \nonumber &=& L(ns,\chi^n)  \frac 1 n  \sum_{l=0}^{n-1} (1-q^{s-1}\chi^{-1}(\varpi)\xi^{l})\xi^{tl}  \Bigl( \sum_{m=0}^{n-1} \bigl(q^{-s}\chi(\varpi)\bigr)^{-m} \xi^{-ml}\Bigr)
 \\ \nonumber &=& L(ns,\chi^n)  \frac 1 n  \sum_{l=0}^{n-1} \sum_{m=0}^{n-1} \bigl(q^{-s}\chi(\varpi)\bigr)^m \xi^{l(t-m)}-q^{-1}\bigl(q^{-s}\chi(\varpi)\bigr)^{m-1}\xi^{l(t+1-m)}.\end{eqnarray}
We now change the order of summation.
\begin{eqnarray} \nonumber &\gamma_{_J}& \! \! \! \! \! \!(1-s,\chi^{-1},\psi,k^{-t}) L(ns,\chi^n)^{-1} \\ \nonumber &=&  \sum_{m=0}^{n-1} \Bigl(\frac 1 n\sum_{l=0}^{n-1} \xi^{l(t-m)}\Bigr)\bigl(q^{-s}\chi(\varpi)\bigr)^m-q^{-1}  \sum_{\widehat{m}=0}^{n-1}\Bigl(\frac 1 n\sum_{l=0}^{n-1} \xi^{l(t-(\widehat{m}-1))}\Bigr)\bigl(q^{-s}\chi(\varpi)\bigr)^{\widehat{m}-1} \\ \nonumber &=&
\sum_{m=0}^{n-1}\bigl(q^{-s}\chi(\varpi)\bigr)^m\delta_{t,m}-q^{-1}  \sum_{\widehat{m}=0}^{n-1}\bigl(q^{-s}\chi(\varpi)\bigr)^{\widehat{m}-1}\delta_{t,\widehat{m}-1}.\end{eqnarray}
Finally, we change the summation index $m=\widehat{m}-1$ in the right summation. This gives
\begin{eqnarray} \nonumber & \gamma_{_J}& \! \! \! \! \! \!\bigl(1-s,\chi^{-1},\psi,k^{-t} \bigr) L\bigl(ns,\chi^n\bigr)^{-1} \\ &=&\sum_{m=0}^{n-1}\bigl(q^{-s}\chi(\varpi)\bigr)^m\delta_{t,m}-q^{-1}  \sum_{{m}=-1}^{n-2}\bigl(q^{-s}\chi(\varpi)\bigr)^m\delta_{t,m}
\\ \nonumber  &=&\bigl(q^{-s}\chi(\varpi)\bigr)^{n-1} \bigl(1-q^{ns-1}\chi^{-n}(\varpi))\delta_{t,n-1}+(1-q^{-1}) \sum_{m=0}^{n-2} \bigl(q^{-s}\chi(\varpi)\bigr)^m\delta_{t,m}.\end{eqnarray}
The formula for the unramified case now follows.
\end{proof}
\begin{prop} \label{odd sha comp} Suppose that $n$ is odd and that  ${\rm gcd}(p,n)=1$. Let $\psi$ be a normalized character of ${\rm F}$, let $\chi$ be a character of ${\rm F}^*$ and let $L$ be the Lagrangian decomposition of ${\rm F}^*/{{\rm F}^*}^d$ given in Example \ref{kp example}. Set $a=k^i, \, b=k^j$.

If $\chi^n$ is ramified then
$$\tau_{_L}(a,b,\chi,s,\psi)= \begin{cases} \epsilon(1-s,\chi^{-1}\eta_{ab},\psi)  & i-j\equiv e(\chi) \, (\operatorname{mod }n) ; \\ \\ 0  & otherwise .\end{cases}$$
If $\chi$ is unramified and $b\neq a^{-1}$ then
$$\tau_{_L}(a,b,\chi,s,\psi)= \begin{cases} \epsilon(1-s,\chi^{-1}\eta_{ab},\psi)  & i-j\equiv 1 \, (\operatorname{mod }n) ; \\ \\ 0  & otherwise .\end{cases}$$
Finally, if  $\chi$ is unramified then
$$\tau_{_L}(a^{-1},a,\chi,s,\psi)= \bigl(q^{-s}\chi(\varpi)\bigr)^{2i-\alpha(i)} \begin{cases} (1-q^{-1}) L(ns,\chi^n)  & 0\leq i \leq {n-1};   \\ \gamma(1-ns,\chi^{-n},\psi) & i=\frac{n-1}{2}  \end{cases}$$
where
$$\alpha(i)=\begin{cases} 0 & 0 \leq i  \leq  \frac{n-1}{2}  ; \\ n & \frac{n+1}{2} \leq i \leq n-1.\end{cases}$$
\end{prop}
\begin{proof} By Theorem \ref{metasha},
$$\tau_{_L}(a,b,\chi,s,\psi)= \gamma_{_J}(1-s,\chi^{-1}\eta_{ab},\psi,ab^{-1}).$$
The first equation follows from \eqref{ram gama} along with the fourth assertion in Lemma \ref{gcd1 stuff}. The second equation follows from \eqref{ram gama} along with the second assertion in Lemma \ref{gcd1 stuff}. The third equation follows from  \eqref{ram gama un}.
\end{proof}
\begin{prop} \label{meta un ram gama}
Suppose that $n\equiv 2 \, (\operatorname{mod }4)$ and that  ${\rm gcd}(p,n)=1$. Let $\psi$ be a normalized character of ${\rm F}$ and let $L$ be the Lagrangian decomposition of ${\rm F}^*/{{\rm F}^*}^d$ given in Example \ref{kp example}.

If $\chi^2$ is ramified then
\begin{equation} \label{un ram gama meta}\widetilde{\gamma}_{_J}(1-s,\chi^{-1},\psi, k^{t})=\begin{cases} \chi(-1)\frac{\epsilon(s+\half,\chi,\psi)}{\epsilon(2s,\chi^2,{\psi_{_2}})} & t\equiv e(\chi) \, (\operatorname{mod }d) ; \\ \\ 0  & otherwise .\end{cases} \end{equation}

If $\chi$ is unramified then
\begin{equation} \label{ram gama un meta1}\widetilde{\gamma}_{_J}(1-s,\chi^{-1},\psi,k^{-2t} )= \bigl(q^{-s}\chi(\varpi)\bigr)^{2t} \begin{cases} (1-q^{-1}) L(ns,\chi^n)  & 0 \leq t \leq d-1, t \neq  \frac {d-1}{2} ; \\  \widetilde{\gamma}(1-ds,\chi^{-d},\psi)& t=\frac {d-1}{2}.  \end{cases} \end{equation}

If $\chi^2$ is unramified and $\chi$ is ramified then
\begin{equation} \label{ram gama un meta2} \widetilde{\gamma}_{_J}(1-s,\chi^{-1},\psi,k^{-(2t+1)} )= \begin{cases} \bigl(q^{-s} \chi(\varpi)\bigr)^{2t}(1-q^{-1}) \chi(-1)\epsilon(s+\half,\chi,\psi)L(ns,\chi^n)  & 0 \leq t \leq d-2 ; \\ \bigl(q^{-s} \chi(\varpi)\bigr)^{d-1}  \widetilde{\gamma}(1-ds,\chi^{-d},\psi)& t= d-1.  \end{cases} \end{equation}
\end{prop}
\begin{proof} Similar to \eqref{beforecommon} we have
$$\widetilde{\gamma}_{_J}(1-s,\chi^{-1},\psi,k^{t} )=
\frac 1 d \sum_{l=0}^{d-1}  \widetilde{\gamma}(1-s,\chi^{-1}\eta_u^l,\psi)\xi^{-tl}.$$
From \eqref{meta gama formula} along with the third assertion in Lemma \ref{gcd1 stuff} and the fact that $-1\in {{\rm F}^*}^d$ it follows that
\begin{equation}\label{metabeforecommon} \widetilde{\gamma}_{_J}(1-s,\chi^{-1},\psi,k^{t} )=
\chi(-1)\frac 1 d \sum_{l=0}^{d-1} \frac{\gamma(s+\half,\chi\eta_u^{-l},\psi)}{\gamma(2s,(\chi\eta_u^{-l})^{2},{\psi_{_2}})}\xi^{-tl}.\end{equation}
Suppose first that $\chi^2$ is ramified, in this case the two $\gamma$-factors in the right hand side of \eqref{metabeforecommon} are $\epsilon$-factors. By  \eqref{epsilon twist} and by the fifth assertion in Lemma \ref{gcd1 stuff} we have
$$\widetilde{\gamma}_{_J}(1-s,\chi^{-1},\psi,k^{t} )=\chi(-1)\frac{\epsilon(s+\half,\chi,\psi)}{\epsilon(2s,\chi^2,{\psi_{_2}})}
\frac 1 d \sum_{l=0}^{d-1} \xi^{l(e(\chi)-t)}=\chi(-1)\frac{\epsilon(s+\half,\chi,\psi)}{\epsilon(2s,\chi^2,{\psi_{_2}})}\delta_{t,e(\chi)}.$$
The theorem for this case is proven.

Suppose now that $\chi$ is unramified. In this case \eqref{meta gama formula} implies that
\begin{equation} \label{nicemetaunram}
\widetilde{\gamma}(1-s,\chi^{-1},\psi)=\frac {(1-q^{-1})+q^{-\half}\bigl(q^{s}\chi^{-1}(\varpi)-q^{-s}\chi(\varpi)\bigr)}{1-q^{-2s}\chi^2(\varpi)}.\end{equation}
Plugging it into \eqref{metabeforecommon} gives
$$\widetilde{\gamma}_{_J}(1-s,\chi^{-1},\psi,k^{-2t} )=\frac 1 d \sum_{l=0}^{d-1}
\frac {(1-q^{-1})+q^{-\half}\bigl(q^{s}\chi^{-1}(\varpi)\xi^l-q^{-s}\chi(\varpi)\xi^{-l}\bigr)}{1-q^{-2s}\chi^2(\varpi)\xi^{-2l}}\xi^{2tl}.$$
Similar to the proof of the  unramified case in Proposition \ref{un ram gama}, we take a common denominator in the right hand side of the last equation and use \eqref{elemntry} (note that since $d$ is odd, $\xi^2$ is a primitive element in $\mu_d$). This gives
\begin{eqnarray} \nonumber &\widetilde{\gamma}_{_J}& \! \! \! \! \! \!(1-s,\chi^{-1},\psi,k^{-2t} ) L(ns,\chi^n)^{-1} \\ \nonumber &=& q^{d-\frac 3 2} \Bigl( \bigl( q^{-s}\chi(\varpi)\bigr)^{-d}-\bigl( q^{-s}\chi(\varpi)\bigr)^d \Bigr)\delta_{2t,-1}+(1-q^{-1}) \sum_{m=0}^{d-1} \bigl(q^{-2s}\chi^2(\varpi)\bigr)^m\delta_{t,m}.\end{eqnarray}
Using \eqref{nicemetaunram} we finsih the proof of the unramified case.

We now prove the case where $\chi$ is ramified and $\chi^2$ is unramified.  In this case,
\begin{equation} \label{nicemetasemiunram}
\widetilde{\gamma}(1-s,\chi^{-1},\psi)= \chi(-1)\epsilon(s+\half,\chi,\psi)\frac{1-q^{2s-1}\chi^{-2}(\varpi)}{1-q^{-2s}\chi^{2}(\varpi)}.
\end{equation}
Combining the second assertion in Lemma \ref{gcd1 stuff} for the $n=2$ case  and \eqref{epsilon twist} we have
\begin{eqnarray} \nonumber  \widetilde{\gamma}_{_J}(1-s,\chi^{-1},\psi,k^{-(1+2t)} ) &=& \chi(-1)\frac 1 d \sum_{l=0}^{d-1}
\epsilon(s+\half,\chi\eta_u^{-l},\psi)\frac{1-q^{2s-1}\chi^{-2}(\varpi)\xi^{2l}}{1-q^{-2s}\chi^{2}(\varpi)\xi^{-2l}}\xi^{(1+2t)l}\\ \nonumber
&=& \chi(-1)\epsilon(s+\half,\chi,\psi)\frac 1 d \sum_{l=0}^{d-1}
\frac{1-q^{2s-1}\chi^{-2}(\varpi)\xi^{2l}}{1-q^{-2s}\chi^{2}(\varpi)\xi^{-2l}}\xi^{2tl}. \end{eqnarray}
By taking a common denominator and by using \eqref{elemntry} we obtain
$$\widetilde{\gamma}_{_J}(1-s,\chi^{-1},\psi,k^{-(1+2t)} )\chi(-1)\epsilon(s+\half,\chi,\psi)^{-1}L(\chi^n,ns)^{-1}=$$
$$\Bigl(\bigl(q^{-2s}\chi^2(\pi)\bigr)^{d-1}L(1-ns,\chi^{-n})^{-1} \Bigr)\delta_{2t+1,-1}+(1-q^{-1}) \sum_{m=0}^{d-2} \bigl((q^{-2s}\chi^2(\varpi)\bigr)^m\delta_{t,m}.$$
Finally, since $d$ is odd and $\chi^2$ is unramified, \eqref{epsilon twist} and \eqref{epsilon old twist} give
$$\epsilon(s+\half,\chi,\psi)=\epsilon(ds+\half-(d-1)s,\chi^d \chi^{1-d},\psi)=\bigl(q^{-s}\chi(\varpi)\bigr)^{1-d}\epsilon(ds+\half,\chi^d,\psi).$$
Using \eqref{nicemetasemiunram} the proof is now completed.
\end{proof}
\begin{prop} \label{even sha comp}  Suppose that $n\equiv 2 \, (\operatorname{mod }4)$  and that  ${\rm gcd}(p,n)=1$. Let $\psi$ be a normalized character of ${\rm F}$, let $\chi$ be a character of ${\rm F}^*$ and let $L$ be the Lagrangian decomposition of ${\rm F}^*/{{\rm F}^*}^d$ given in Example \ref{kp example}. Set $a=k^i, \, b=k^j$.

If $\chi^n$ is ramified then
$$\tau_{_L}(a,b,\chi,s,\psi)= \begin{cases} \chi(-1)\frac{\epsilon(s+\half,\chi\eta_{ab}'^{-1},\psi)}{\epsilon(2s,\chi^2 \eta_{ab}'^{-2},{\psi_{_2}})} & i-j\equiv e(\chi) \, (\operatorname{mod }d) ; \\ \\ 0  & otherwise .\end{cases}$$
If $\chi^n$ is unramified and $b\neq a^{-1}$ then
$$\tau_{_L}(a,b,\chi,s,\psi)= \begin{cases} \chi(-1)\frac{\epsilon(s+\half,\chi\eta_{ab}'^{-1},\psi)}{\epsilon(2s,\chi^2 \eta_{ab}'^{-2},{\psi_{_2}})} & i-j\equiv 1 \, (\operatorname{mod }d) ; \\ \\ 0  & otherwise .\end{cases}$$
If  $\chi$ is unramified then
$$\tau_{_L}(a^{-1},a,\chi,s,\psi)=\bigl(q^{-s}\chi(\varpi)\bigr)^{2i} \begin{cases} (1-q^{-1}) L(ns,\chi^n)  & 0 \leq i \leq d-1, i \neq  \frac {d-1}{2} ; \\  \widetilde{\gamma}(1-ds,\chi^{-d},\psi)& i=\frac {d-1}{2}.\end{cases}$$
If  $\chi^2$ is unramified and $\chi$ is  ramified then
$$\tau_{_L}(a^{-1},a,\chi,s,\psi)=  \begin{cases} \bigl(q^{-s} \chi(\varpi)\bigr)^{d-1} \widetilde{\gamma}(1-ds,\chi^{-d},\psi) & i=\frac{d-1}{2}  ; \\ \bigl(q^{-s} \chi(\varpi)\bigr)^{2i-1+\beta(i)}  (1-q^{-1})\chi(-1)\epsilon(s+\half,\chi,\psi)L(ns,\chi^n) & 0 \leq i \leq d-1, \,  i \neq  \frac{d-1}{2} \ \end{cases}$$
where $$\beta(i)=\begin{cases} d & 0 \leq i \leq \frac{d-3}{2}  ; \\ -d & \frac{d+1}{2} \leq i \leq d-1.\end{cases}$$
\end{prop}
\begin{proof}
By Theorem \ref{metasha},
$$\tau_{_L}(a,b,\chi,s,\psi)= \widetilde{\gamma}_{_J}(1-s,\chi^{-1}\eta_{ab},\psi,ab^{-1}).$$
The first equation follows from \eqref{un ram gama meta} along with the fourth assertion in Lemma \ref{gcd1 stuff}. The second equation follows from \eqref{un ram gama meta} along with the second assertion in Lemma \ref{gcd1 stuff}. The third and fourth equations follow from  \eqref{ram gama un meta1} and \eqref{ram gama un meta2} respectively.
\end{proof}
\section{Plancherel measure }\label{planme}
\subsection{Definition}
Let $\sigma$ be a genuine smooth irreducible  representation of  $\widetilde{\rm{H}}$. Since the isomorphism class of  $\sigma$ is determined by its central character, it follows from \eqref{sigmainter} that $\sigma_s \cong (\sigma_s)^w$ only for a (possibly empty) finite set of values of $q^{-ds}$. By Part 2 of Theorem 7.1 in \cite{Mc}, this implies that for almost all values for $q^{-ds}$, the commuting algebra of ${\rm I}(\sigma_s)$ is 1 dimensional. Thus, following Shahidi, see Section 5.3 in \cite{Shabook}, we define  $\mu_n(\sigma,s)$, the Plancherel measure associated with $\sigma$ by
\begin{equation} \label{pandef} {\rm A}_{w^{-1}}\bigl((\sigma_s)^w\bigr )\circ {\rm A}_{w}\bigl(\sigma_s \bigr)=\mu_n(\sigma,s)^{-1}Id. \end{equation}
Note that $\mu_n(\sigma,s)$ is a rational function in $q^{-s}$. The Plancherel measure depends weakly on $\psi$ via the normalization of the Haar measure in the intertwining operators. We suppress this dependence. We note here that since
$w^{-1}=\bigl(-I_2,(-1,-1)_n\bigr)w$
and since $\bigl(-I_2,(-1,-1)_n\bigr) \in \widetilde{C}$ we obtain
\begin{equation} \label{pandef modi} {\rm A}_{w}\bigl((\sigma_s)^w\bigr )\circ {\rm A}_{w}\bigl(\sigma_s \bigr)=\mu_n(\chi_\psi,s)^{-1}(\chi_\sigma)^{-1}\bigl(-I_2,(-1,-1)_n\bigr).\end{equation}
\subsection{Relations with Shahidi local coefficients matrices}
By Remark \ref{plan is lc}, it follows from \eqref{pandef} that
\begin{equation} \label{gensha} {\rm A}_w^\psi \bigl(\sigma_s \bigr) \circ  {A^{\psi}_{w^{-1}}}\bigl((\sigma_s)^w \bigr)=\mu_n(\sigma,s)^{-1}.\end{equation}
This is a metaplectic analog for the relation between Shahidi local coefficients and Plancherel measures, see Corollary 5.3.1 in \cite{Shabook}.

Since we realize ${\rm I}\bigl(\sigma_s \bigr)$ as  ${\rm I}\bigl(\chi'_{\psi,s}\bigr)$ it is convenient to use the notation $$\mu_n(\chi'_\psi,s)=\mu_n(\sigma,s).$$ However, it is important to note that $\mu_n(\chi'_\psi,s)$ is an invariant of $\sigma$. It is independent of the choice of a maximal abelian subgroup and of the chosen extension of $\chi_\sigma$. In particular
\begin{equation}\label{aver}
 \mu_n(\chi'_\psi,s)^{-1}={(\# \overline{K})}^{-1}\sum_{k \in \overline{K}} \mu_n \bigl((\chi'\eta_k)_\psi,s \bigr)^{-1}.
 \end{equation}
Remark \ref{plan is lc} and \eqref{pandef} imply that for any $\xi \in \operatorname{Wh}_\psi \bigl({\rm I}\bigl({(\chi'^{-1}_\psi)}_{-s}\bigr)\bigr)$ we have
\begin{equation}  \label{central -1} { A^{\psi}_{w}}\bigl(\chi'_\psi,s\bigr) \circ {A^{\psi}_{w}}\bigl({(\chi'^{-1})}\! _\psi,-s\bigr)(\xi)=\mu_n(\chi_\psi,s)^{-1}(\chi_\sigma)^{-1}\bigl(-I_2,(-1,-1)_n\bigr)\xi.\end{equation}
In particular,  for $\xi=\lambda_{1,\chi^{-1},\psi,-s}$, \eqref{mykpmatrix} implies
\begin{equation}\label{plan and sum}  \mu_n(\chi_\psi,s)^{-1}=(-1,-1)_n \, \chi_\psi(-1)\sum_{k \in \overline{K}} \tau_{_L}(1,k,\chi_{},s,\psi)\tau_{_L}(k^{-1},1,\chi^{-1}_{},-s,\psi).\end{equation}
Note that in the case where $n \leq 2$, there is only one summand in the right hand side of \eqref{plan and sum}. The  formula for $\mu_1(\chi_\psi,s)$ is well known. The formula for $\mu_2(\chi_\psi,s)$ follows easily from \eqref{norweildef} along with our previous work \cite{Sz2}. For the sake of completeness we now give these formulas along with proofs. In the $n=1$ case, \eqref{plan and sum} along with Theorem \ref{metasha} and \eqref{epsilon twist and inv} give
\begin{equation} \label{plan1}\mu_1(\chi,s)^{-1}= \chi(-1)\gamma(1-s, \chi^{-1},\psi)\gamma( 1+s,\chi,\psi)= q^{e(\psi)-e(\chi)}\frac{L (s,\chi )L (-s,\chi^{-1}\bigr)}{L (1-s,\chi^{-1} )L (1+s,\chi)}. \end{equation}
Combining the same arguments with \eqref{epsilon twist and inv} and \eqref{lasttwist} give
 \begin{eqnarray} \label{plan2}\mu_2(\chi_\psi,s)^{-1} &=& (-1,-1)_2 \chi_\psi(-1)\widetilde{\gamma}(1-s, \chi^{-1},\psi)\widetilde{\gamma}( 1+s,\chi,\psi)\\ \nonumber &=& q^{e(\psi_2)-e(\chi^2)}\frac{L (2s,\chi^2 )L (-2s,\chi^{-2}\bigr)}{L (1-2s,\chi^{-2} )L (1+2s,\chi^2)}.\end{eqnarray}
\subsection{Averaging formulas} \label{avformula}
\begin{thm} \label{new plan}  Let $\sigma$ be a genuine smooth irreducible  representation of  $\widetilde{\rm{H}}$ with a central character  $\chi_\sigma$. Let $\chi$ be a character of ${\rm F}^*$ such that $\chi''_\psi=\chi_\sigma$.
\begin{equation}
 \label{new plan form} \mu_n(\sigma,s)^{-1}=[{\rm F}^*:{{\rm F}^*}^d]^{-1}  \sum_{\eta \in \widehat{{\rm F}^*/ {{\rm F}^*}^d}} \mu_{n/d} \bigl((\chi\eta)_\psi,s \bigr)^{-1}. \end{equation}
\end{thm}
\begin{proof}
We first prove \eqref{new plan form} under the assumption that $n$ is odd. We note that for any $j \in J$, $(\chi \eta_j)(-1)=\chi(-1)$. Thus, by \eqref{plan and sum} and by the left equality in \eqref{plan1} it is sufficient to show that
\begin{eqnarray} \label{to show}
&& \sum_{k \in \overline{K}} \tau_{_L}(1,k,\chi_{},s,\psi)\tau_{_L}(k^{-1},1,\chi^{-1}_{},-s,\psi) \\ \nonumber
&&=[{\rm F}^*:{{\rm F}^*}^d]^{-1} \sum_{x\in  {\rm F}^*/{{\rm F}^*}^d} \gamma (1-s,\chi^{-1}\eta^{-1}_{x},\psi )\gamma (1+s,\chi\eta_{x},\psi).
\end{eqnarray}
By Theorem \ref{metasha} and by \eqref{part gamma def}
\begin{eqnarray} \nonumber  &&\sum_{k \in \overline{K}} \tau_{_L}(1,k,\chi_{},s,\psi)\tau_{_L}(k^{-1},1,\chi^{-1}_{},-s,\psi) \\ \nonumber &&=[{\rm F}^*:{{\rm F}^*}^d]^{-1}  \sum_{k\in \overline{K}}  \sum_{h\in {\overline{J}}} \sum_{j\in {\overline{J}}} \gamma (1-s,\chi^{-1}\eta_{jk},\psi )\gamma (1+s,\chi\eta_{hk^{-1}},\psi)\eta_{jh}(k).\end{eqnarray}
We now change a summation index, $j \mapsto h^{-1}j$. This gives
\begin{eqnarray} \nonumber  && \sum_{k \in \overline{K}} \tau_{_L}(1,k,\chi_{},s,\psi)\tau_{_L}(k^{-1},1,\chi^{-1}_{},-s,\psi)\\ \nonumber &&=[{\rm F}^*:{{\rm F}^*}^d]^{-1}\sum_{k\in \overline{K}}  \sum_{h\in {\overline{J}}} \sum_{j\in {\overline{J}}} \gamma (1-s,(\chi\eta_{hk^{-1}})^{-1}\eta_{g},\psi )\gamma (1+s,\chi\eta_{hk^{-1}},\psi )\eta_{j}(k)\\ \nonumber &&=[{\rm F}^*:{{\rm F}^*}^d]^{-1}  \sum_{x\in {{\rm F}^*}/{{\rm F}^*}^d} \sum_{j\in {\overline{J}}} \gamma \bigl(1-s,(\chi\eta_{x})^{-1}\eta_j,\psi \bigr)\gamma (1+s,\chi\eta_{x},\psi )\eta_{j}(x^{-1}).\end{eqnarray}
By \eqref{aver} along with \eqref {plan and sum} we now conclude that
\begin{eqnarray}\nonumber
&& \sum_{k \in \overline{K}} \tau_{_L}(1,k,\chi_{},s,\psi)\tau_{_L}(k^{-1},1,\chi^{-1}_{},-s,\psi)\\ \nonumber &&=[{\rm F}^*:{{\rm F}^*}^d]^{-\frac 3 2}\sum_{y \in  \overline{K}}  \sum_{x\in {{\rm F}^*}/{{\rm F}^*}^d} \sum_{j\in {\overline{J}}} \gamma \bigl(1-s,(\chi\eta_{xy})^{-1}\eta_j,\psi \bigr)\gamma (1+s,\chi\eta_{xy},\psi \bigr)\eta_{j}(x^{-1}).
\end{eqnarray}
We finally change another summation index, $x \mapsto y^{-1}x$ and change the order of summation.
\begin{eqnarray} \nonumber
&& \sum_{k \in \overline{K}} \tau_{_L}(1,k,\chi_{},s,\psi)\tau_{_L}(k^{-1},1,\chi^{-1}_{},-s,\psi)\\ \nonumber
&&=[{\rm F}^*:{{\rm F}^*}^d]^{-\frac 3 2}\sum_{y \in  \overline{K}}  \sum_{x\in {{\rm F}^*}/{{\rm F}^*}^d} \sum_{j\in {\overline{J}}} \gamma \bigl(1-s,(\chi\eta_{x})^{-1}\eta_j,\psi \bigr)\gamma (1+s,\chi\eta_{x},\psi )\eta_{j}(x^{-1}y)\\ \nonumber &&=[{\rm F}^*:{{\rm F}^*}^d]^{-\frac 3 2}\sum_{x\in {{\rm F}^*}/{{\rm F}^*}^d} \sum_{j\in {\overline{J}}}   \gamma \bigl(1-s,(\chi\eta_{x})^{-1}\eta_j,\psi \bigr)\gamma (1+s,\chi\eta_{x},\psi \bigr)\eta_{j}(x^{-1}) \sum_{y \in  \overline{K}} \eta_j(y).\end{eqnarray}
Since for $j \in {J}$, $$\sum_{y \in  \overline{K}} \eta_j(y)=\begin{cases} [{\rm F}^*:{{\rm F}^*}^n]^{1/2}  &  j \in {{\rm F}^*}^d; \\ 0  &   \operatorname{otherwise}, \end{cases}$$
Equation \eqref{to show} now follows. The proof for the case $n \equiv 2 \, (\operatorname{mod }4)$ follows in the same way, replacing $\gamma$ by $\widetilde{\gamma}$ and $\eta$ by $\eta'$.
\end{proof}
Recall that $\widetilde{C}^{(n)}$ is the center of $\widetilde{\rm{H}}^{(n)}$, the inverse image of ${\rm H}$ inside $\widetilde{\rm{G}}^{(n)}$. From Lemma \ref{center and max} it follows that if $m$ divides $n$ and ${\rm gcd}(n,2)={\rm gcd}(m,2)$ then $\widetilde{C}^{(n)}$ is a subgroup of   $\widetilde{C}^{(m)}$ although $\widetilde{\rm{H}}^{(n)}$ is not a subgroup of   $\widetilde{\rm{H}}^{(m)}$. We say that a genuine smooth irreducible representation $\sigma$ of $\widetilde{\rm{H}}^{(n)}$  and a genuine smooth irreducible  representation $\pi$ of  $\widetilde{\rm{H}}^{(m)}$ are related if $\chi_\sigma$ is the restriction of $\chi_\pi$. Denote $$c=\begin{cases} m &  m \, \operatorname{is} \,  \operatorname{odd }; \\  \frac {m} {2}  &  m \, \operatorname{is} \,  \operatorname{even}  \end{cases}$$
 and observe that our assumptions on $m$ and $n$ implies that $c$ divides $d$. From the description of the genuine characters of $\widetilde{C}^{(n)}$ given in Section \ref{modelsection} one concludes that the set $E_m(\sigma)$ of  genuine smooth irreducible representations of $\widetilde{\rm{H}}^{(m)}$ related to $\sigma$ is in bijection with the set of characters of ${{\rm F}^*}^c$ which are trivial on ${{\rm F}^*}^d$, namely with the set
 $$\{{\eta^o_x} \mid x \in {\rm F}^*/{{\rm F}^*}^{d} \}$$
 where ${\eta^o_x}$ is the restriction of $\eta_x$ to  ${{\rm F}^*}^c$. Note now that ${\eta^o_x}={\eta^o_y}$ if and only if ${(xy^{-1})}^{c} \in {{\rm F}^*}^d$. Thus,  $E_m(\sigma)$ is parameterized by
 $$\{{\eta^o_a} \mid x \in ({\rm F}^*/{{\rm F}^*}^{d}) /  ({\rm F}^*/{{\rm F}^*}^{c}) \}$$ which is naturally identified with $\widehat{({\rm F}^*/{{\rm F}^*}^{d})} / \widehat{ ({\rm F}^*/{{\rm F}^*}^{c})} .$ Therefore, by writing the right hand side of \eqref{new plan form} as
$${\bigl(\# E_m(\sigma )\bigr)}^{-1} \! \!   \! \!\sum_{\pi  \in \widehat{({\rm F}^*/{{\rm F}^*}^{d})} / \widehat{ ({\rm F}^*/{{\rm F}^*}^{c})} }\Biggr({[{{\rm F}^*}:{{\rm F}^*}^c]}^{-1}  \! \! \sum_{\eta \in \widehat{{\rm F}^*/{{\rm F}^*}^{c}}} \! \! \mu_{m/c} \bigl((\chi\eta \pi)_\psi,s \bigl) \Biggr)$$
we have proven the following
\begin{cor} \label{plan as av}
$$\mu_n(\sigma,s)^{-1}={\bigl(\# E_m(\sigma)\bigr)}^{-1} \sum_{\pi \in  E_m(\sigma) } \mu_m(\pi,s)^{-1}$$
\end{cor}
Observe that Corollary \ref{plan as av} generalizes \eqref{new plan form}. It presents the Plancherel measure associated a  genuine smooth irreducible representation $\sigma$ of $\widetilde{\rm{H}}^{(n)}$  as the harmonic mean of the  Plancherel measures associated with the genuine smooth irreducible representations of $\widetilde{\rm{H}}^{(m)}$ related to $\sigma$.
\subsection{$T(\sigma,s,\psi)$, $D(\sigma,s,\psi)$ and the Plancherel measure} \label{DTandPlan}
Recall that  we have defined $T(\sigma,s,\psi)$ and $D(\sigma,s,\psi)$ respectively to be the trace and determinant  of an Slcm associated with $\sigma$ and $\psi$ and that these are well defined invariants of $\sigma$ and $\psi$. For $n \leq 2$, the Slcms are scalars. Thus, for these cases the relation between $T(\sigma,s,\psi)$, $D(\sigma,s,\psi)$ and the Plancherel measure is given in \eqref{plan1} and \eqref{plan2}. We generalize these relations for all $n$.

We first note that
\begin{equation} \label{new new  plan form}
 D(\sigma,s,\psi)D(\sigma^w,-s,\psi)=\Bigl(\mu_n(\sigma,s)\chi_\sigma\bigl(-I_2,(-1,-1)_n\bigr)\Bigr)^{-d\ab d \ab^{-\half}}. \end{equation}
Indeed, this follows at once from \eqref{central -1} and from \eqref{whidim}. We now prove the following
\begin{proposition} \label{traceandplan}Let $\sigma$ be a genuine smooth irreducible  representation of  $\widetilde{\rm{H}}$ with a central character  $\chi_\sigma$. For odd $n$ we have
\begin{equation} \label{traceandplan1} [{\rm F}^*:{{\rm F}^*}^d]^{-1}\sum_{a \in {{\rm F}^*/{{\rm F}^*}^d}}\ab a \ab^{-1}  T(\sigma,s,\psi_a) T(\sigma^w,-s,\psi_a)=\chi_\sigma(-I_2,1)
\mu_n(\sigma,s)^{-1}.\end{equation}
For  $n \equiv 2\, (\operatorname{mod }4)$ we have
\begin{eqnarray}\label{traceandplan2} &&[{\rm F}^*:{{\rm F}^*}^d]^{-1} \! \! \! \sum_{a \in {{\rm F}^*/{{\rm F}^*}^d}} \! \! \ab a \ab^{-1} (a,-1)_2 \, T\bigr( (\cdot,a)_2 \otimes \sigma,s,\psi_a \bigr) T((\cdot,a)_2 \otimes \sigma^w,-s,\psi_a)\\ \nonumber &&= \chi_\sigma\bigl(-I_2,(-1,-1)_n \bigr) \mu_n(\sigma,s)^{-1}.\end{eqnarray}
\end{proposition}
\begin{proof} Let $\chi$ be a character of ${\rm F}^*$ such that $\chi''_\psi=\chi_\sigma$. We first prove the theorem for odd $n$. Using Corollary \ref{more inv} and Remark \ref{for T prop} along with Equation \eqref{changepsi} we obtain
\begin{eqnarray} \nonumber &&T(\sigma,s,\psi_a) T(\sigma^w,-s,\psi_a) \\ \nonumber &&= [{\rm F}^*:{{\rm F}^*}^d]^{-1}\ab a \ab \sum_{c \in {{\rm F}^*/{{\rm F}^*}^d}} \sum_{b \in {{\rm F}^*/{{\rm F}^*}^d}}\eta_{cb}(a)\gamma(1-s,\chi^{-1}\eta_c,\psi)  \gamma(1+s,\chi\eta_b,\psi).\end{eqnarray}
This shows that $$a\mapsto F_{s,\sigma,\psi}(a)= \ab a \ab^{-1}  T(\sigma,s,\psi_a) T(\sigma^w,-s,\psi_a)$$ is a well defined function on ${{\rm F}^*/{{\rm F}^*}^d}$.
By changing summation index $b \mapsto e=bc$ and then by changing summation order we find the Fourier expansion of $F_{s,\sigma,\psi}(a)$:
$$ F_{s,\sigma,\psi}(a)=  \sum_{e \in {{\rm F}^*/{{\rm F}^*}^d}}\eta_e(a)\widehat{F}_{s,\sigma,\psi}(e),$$
where
$$\widehat{F}_{s,\sigma,\psi}(e)= [{\rm F}^*:{{\rm F}^*}^d]^{-1}\sum_{c \in {{\rm F}^*/{{\rm F}^*}^d}}\gamma(1-s,\chi^{-1}\eta_c,\psi)  \gamma(1+s,\chi\eta_{ec^{-1}},\psi).$$
The $e=1$ term in this expansion is

$$[{\rm F}^*:{{\rm F}^*}^d]^{-1}\sum_{a \in {{\rm F}^*/{{\rm F}^*}^d}} F_{s,\sigma,\psi}(a)=\widehat{F}_{s,\sigma,\psi}(1).$$
We have shown:
$$\sum_{a \in {{\rm F}^*/{{\rm F}^*}^d}} F_{s,\sigma,\psi}(a)=\sum_{c \in {{\rm F}^*/{{\rm F}^*}^d}}\gamma \bigl(1-s,(\chi\eta_c)^{-1},\psi \bigr)  \gamma \bigl(1+s,\chi\eta_c,\psi \bigr).$$
Comparing this with \eqref{plan and sum} and \eqref{to show} the result for odd $n$ now follows.

Assume now that $n \equiv 2\, (\operatorname{mod }4)$. In this case Corollary \ref{more inv} along with Remark \ref{for T prop} imply that
$$T\bigl((\cdot,a)_2 \otimes \sigma,s,\psi_a\bigr)=\sum_{\eta \in \widehat{{\rm F}^*/{{\rm F}^*}^d}} \widetilde{\gamma}(1-s,\chi^{-1}\eta,\psi_a).$$
Using \eqref{metachangepsi} and \eqref{weilsqyare} we now deduce that
$$a \mapsto (a,-1)_2\ab a \ab^{-1}  T(\sigma,s,\psi_a) T(\sigma^w,-s,\psi_a)=$$ $$\sum_{c \in {{\rm F}^*/{{\rm F}^*}^d}} \sum_{b \in {{\rm F}^*/{{\rm F}^*}^d}}\eta_{bc}(a)\widetilde{\gamma}(1-s,\chi^{-1}\eta_b,\psi) \widetilde{ \gamma}(1+s,\chi^{-1}\eta_c,\psi)$$ is a well defined function on ${{\rm F}^*/{{\rm F}^*}^d}$. The rest of the proof goes word for word as the proof in the odd case.
\end{proof}
\begin{remark} Observe that the factor $\ab a \ab^{-1}$ that appears in the left hand side of \eqref{traceandplan1} and  \eqref{traceandplan2} compensates for the dependence of the normalization of the intertwining operators on $\psi_a$.
\end{remark}
\subsection{A reducibility result}
\begin{prop} \label{irrcor} Let $\sigma$ be a  unitary genuine smooth irreducible representation of  $\widetilde{\rm{H}}$. Denote its central character by $\chi_\sigma$. Then, ${\rm I}(\sigma)$ is reducible if and only if $n$ is odd and the restriction of $\chi_\sigma$ to $sec\bigl(\widetilde{C}\bigr) \cong {{\rm F}^*}^n$ is a non-trivial quadratic character. In this case $\sigma$ is a direct sum of two non-isomorphic irreducible representations.
\end{prop}
\begin{proof}
From the Knapp-Stein dimension Theorem extended by Savin in the Appendix of \cite{TA} to a maximal parabolic induction on metaplectic groups, it follows that given  that $\sigma$ is unitary, ${\rm I}(\sigma)$ is reducible if and only if $\sigma \cong \sigma^w$  and $\mu_n(\sigma,s)^{-1}$ is analytic at $s=0$. In this case ${\rm I}(\sigma)$ is the sum of two non-isomorphic irreducible representations. Denote $\chi_\sigma=\chi''_\psi$. Since $\chi_{\sigma^w}={(\chi''^{-1})}_\psi$ it follows that $\sigma \cong \sigma^w$ is equivalent to $\chi''^2=1$. Let $\chi$ be any extension of $\chi''$ to ${\rm F}^*$. Since $\chi''^2=1$ one concludes that the order of $\chi$ divides $2d$.

Suppose first that $n$ is even. In this case $2d=n$. From \eqref{plan2} it follows that there exist $ [{\rm F}^*:{{\rm F}^*}^2] $ identical summands in the right hand side of \eqref{new plan form} which have a pole of order 2 at $s=0$. The rest of the summands are analytic. Thus, $\sigma \cong \sigma^w$ implies that $\mu_n(\sigma,s)^{-1}$ has a pole at $s=0$.

Suppose now that $n$ is odd. If $\chi$ is of order dividing $n$ then $\chi_\sigma$ is trivial and from \eqref{plan1} it follows that exactly one of the summands in the right hand side of \eqref{new plan form} has a pole of order 2 at $s=0$. Thus, $\mu_n(\sigma,s)^{-1}$ has a pole at $s=0$ in this case. However if $\chi''^2=1$ and $\chi''$ is non-trivial then $\chi$ must be of order $2n$. In this case all the summands in the right hand side of \eqref{new plan form} are analytic. This implies that  $\mu_n(\sigma,s)^{-1}$ is analytic at $s=0$.
\end{proof}
\begin{remark} Using a different argument, it was shown in Theorem 5.3.6 of \cite{K} that if $n$ divides $q-1$ then the (unique) reducible unitary unramified genuine principal series representation is a direct sum of two non-isomorphic irreducible representations.
\end{remark}
\subsection{An explicit formula} \label{expfor}
In the case where ${\rm gcd}(n,p)=1$ it was proven in Theorem 5.1 of \cite{GoSz} that for $\sigma$ and $\chi$ as before,
\begin{equation} \label{myold result}\mu_n(\sigma,s)^{-1}= q^{e(\psi)-e(\chi^n)}\frac{L \bigl(ns,\chi^n \bigr)L \bigl(-ns,\chi^{-n}\bigr)}{L \bigl(1-ns,\chi^{-n} \bigr)L \bigl(1+ns,\chi^{n}\bigr)}.\end{equation}
In Section 8.5 of \cite{Gao17} Gao utilized a global-local argument and proved that for any $n$ and for any $p$-adic field ${\rm F}$ containing $\mu_n$,  the zeros and poles of $\mu_n(\sigma,s)^{-1}$ are the identical to the zeros and poles of the right hand side of \eqref{myold result}. Combining  this result with \eqref{new plan form} we now give an explicit formula for $\mu_n(\sigma,s)^{-1}$ provided that $n \not \equiv 0 \, (\operatorname{mod }4)$.
\begin{thm} Let $\sigma$ be a genuine smooth irreducible  representation of  $\widetilde{\rm{H}}$ with a central character  $\chi_\sigma$. Let $\chi$ be a character of ${\rm F}^*$ such that $\chi''_\psi=\chi_\sigma$. Suppose that $n \not \equiv 0 \, (\operatorname{mod }4)$. Then,
$$\mu_n(\sigma,s)^{-1}= c(\sigma)\frac{L \bigl(ns,\chi^n \bigr)L \bigl(-ns,\chi^{-n}\bigr)}{L \bigl(1-ns,\chi^{-n} \bigr)L \bigl(1+ns,\chi^{n}\bigr)}$$
where $c(\sigma)$ is a positive constant given by
$$c(\sigma)=q^{e(\psi_{n/d})}\begin{cases} [{\rm F}^*:{{\rm F}^*}^d]^{-1}\sum_{\eta \in \widehat{{\rm F}^*/ {{\rm F}^*}^d}} q^{-e(\chi^{n/d}\eta^{n/d})} &  \chi^n \operatorname{is} \,  \operatorname{ramified },  \\ \ab d \ab  &  \chi^n \operatorname{is} \,  \operatorname{unramified. } \end{cases}$$
\end{thm}
\begin{proof} We only give a proof when $n$ is odd. The case $n \equiv 2 \, (\operatorname{mod }4)$ is proven by a similar argument. The case where $\chi^n$ is ramified follows at once from \eqref{new plan form} and \eqref{plan1}.

We move to the unramified case. Since we can twist $\chi$ by $\eta_a$ where $a \in {\rm F}^*$ and shift the complex parameter $s$ we may assume the $\chi=\chi^o$ is the trivial character. Using \eqref{new plan form} and \eqref{plan1} again we have
$$\mu_n(\sigma,s)^{-1}=n^{-2} \ab n \ab \sum_{a \in {{\rm F}^*/ {{\rm F}^*}^n}} \mu_{1} \bigl((\eta_a)_\psi,s \bigr)^{-1}=q^{e(\psi)}n^{-2} \ab n \ab  \times$$ $$\Bigl( \sum_{\substack{a \in {{\rm F}^*}/ {{\rm F}^*}^n\\ \eta_a \text{ is ramified}}}q^{-e(\eta_a)}+\sum_{\substack{a \in {{\rm F}^*}/ {{\rm F}^*}^n\\ \eta_a \text{ is unramified}}}\frac{L (s,\eta_a )L (-s,\chi^{-1}\eta_a)}{L (1-s,\eta_a^{-1} )L (1+s,\eta_a)}\Bigr).$$
For any $p$-adic field, the group of unramified characters whose order divides $n$ is a cyclic group of order $n$ generated by $x \mapsto \ab x\ab^c$ where $q^{-c}$ is a primitive $n^{\text{th}}$ root of 1. Thus,
$$\sum_{\substack{a \in {{\rm F}^*}/ {{\rm F}^*}^n\\ \eta_a \text{ is unramified}}}\frac{L (s,\eta_a )L (-s,\eta_a^{-1})}{L (1-s,\eta_a^{-1} )L (1+s,\eta_a)}=M(q^{-s})$$
where $M$ is a rational function, independent of the $p$-adic field ${\rm F}$. From the above it is also clear that $$\# \{a \in {\rm F}^*/{{\rm F}^*}^n \mid \eta_a \text{ is ramified} \}= [{\rm F}^*:{{\rm F}^*}^n]-n$$
By the second assertion of Lemma \ref{gcd1 stuff}, if ${\rm gcd}(p,n)=1$ then the conductor of any ramified character in $\widehat{{{\rm F}^*/ {{\rm F}^*}^n}}$ is 1. Therefore, we write
$$ \sum_{\substack{a \in {{\rm F}^*}/ {{\rm F}^*}^n\\ \eta_a \text{ is ramified}}}q^{-e(\eta_a)}=(n^2-n)q^{-1}+A({\rm F})$$
where $A({\rm F})$ is a non-negative constant depending on ${\rm F}$ such that $A({\rm F})=0$ provided that ${\rm gcd}(p,n)=1$. Using this notation we have
$$\mu_n(\sigma,s)^{-1}=q^{e(\psi)}n^{-2} \ab n \ab \bigl((n^2-n)q^{-1}+A({\rm F})+M(q^{-s})\bigr).$$
Comparing \eqref{myold result} and  \eqref{new plan form} we observe that
$$(n^2-n)q^{-1}+M(q^{-s})=n^2\frac{L (ns,\chi^o )L (-ns,\chi^o)}{L (1-ns,\chi^o )L (1+ns,\chi^o)}$$
(this may also be verified by a direct elementary computation similar to those given in Section \ref{uram examples}). Finally we obtain
$$\mu_n(\sigma,s)^{-1}=q^{e(\psi)}\ab n \ab \frac{L (ns,\chi^o) L (-ns,\chi^o)}{L (1-ns,\chi^o)L (1+ns,\chi^o)}+q^{e(\psi)}n^{-2} \ab n \ab A({\rm F}).$$
The result of Gao mentioned above now guarantees that  $A({\rm F})=0$ for all the $p$-adic fields in discussion. The theorem for the odd cases is now proven.
\end{proof}
We note here that a by-product of our last proof is the following curious identity which holds for any $p$-adic field ${\rm F}$ containing the full group of $n^{\operatorname{th}}$ roots of 1, where $n$ is odd.
$$\sum_{\eta \in \widehat{{\rm F}^*}, \, \eta^n=1} q^{-e(\eta)}=n(n-1)q^{-1}+n.$$
\subsection{A remark on the $n\equiv 0 \, (\operatorname{mod }4)$ case}
In this paper we did not compute the Slcms for the case where $n\equiv 0 \, (\operatorname{mod }4)$. The technical reason for this is that our method depends on explicit parametrization of the genuine characters of a maximal abelian subgroup of $\widetilde{\rm{H}}$. We could not complete this task when  both $n\equiv 0 \, (\operatorname{mod }4)$ and $p$ divides $n$. Through this paper we have used the fact $[{\rm F}^*:{{\rm F}^*}^d]$ is odd
wherever $n \not \equiv 0 \, (\operatorname{mod }4)$. This does not hold for the $n\equiv 0 \, (\operatorname{mod }4)$ case. In light of Sections 3 and 4 of \cite{GSS} we believe that  this difference is not only technical.

In \cite{GoSz} we have studied the $n\equiv 0 \, (\operatorname{mod }4)$ case under the assumption that ${\rm gcd}(p,n)=1$. In this case ${\rm F}^*/ {{\rm F}^*}^d$ is simple enough so we were able to parameterize the genuine characters of the maximal abelian subgroup of  $\widetilde{\rm{H}}$ arising from the Lagrangian decompression given in Example \ref{kp example}, see the second part of Lemma 1.5 and Section 3.2 in \cite{GoSz}. As a result, an object closely related to an Slcm was computed, see Lemma 4.3 in  \cite{GoSz}. We have used this computation to prove that \eqref{myold result} holds also when $n\equiv 0 \, (\operatorname{mod }4)$  (in which case $\widetilde{\rm{G}}^{(n)}$ splits over $\widetilde{C}^{(n)}$ via the section $sec$) provided that ${\rm gcd}(n,p)=1$. By a similar computations to those presented in Section \ref{uram examples}, one shows that under the assumption ${\rm gcd}(n,p)=1$, \eqref{new plan form} is equivalent to \eqref{myold result}. Thus, \eqref{new plan form} also holds in the case $n \equiv 0\, (\operatorname{mod }4)$ provided that ${\rm gcd}(n,p)=1$. Same is true for Corollary \ref{plan as av}. Using \eqref{myold result} we have shown in Theorem 5.2 of \cite{GoSz} that no reducibilities occur on the unitary axis when $n$ is even and  ${\rm gcd}(n,p)=1$. This result holds also when  ${\rm gcd}(n,p)>1$. Indeed, the $n\equiv 2 \, (\operatorname{mod }4)$ was proven here as Proposition  \ref{irrcor}. The $n\equiv 0 \, (\operatorname{mod }4)$ may be proven by a similar argument to the one used in Theorem 5.2 of \cite{GoSz} replacing \eqref{myold result} with the weaker result of Gao mentioned in Section \ref{expfor} above. Last, \eqref{new new  plan form} holds for $n\equiv 0 \, (\operatorname{mod }4)$ as the argument we used to prove it carries over to this case as well.

We did not check the validity of Proposition \ref{traceandplan} when  $n\equiv 0 \, (\operatorname{mod }4)$ but we believe that assuming that  ${\rm gcd}(n,p)=1$ one can test it by modifying the computations in  \cite{GoSz}.
\bibliography{dani}
\bibliographystyle{acm}

\end{document}